\documentclass[11pt]{article}

\usepackage[margin=2.5cm]{geometry}
\usepackage[utf8]{inputenc}
\usepackage[T1]{fontenc}
\usepackage{amsmath,amsthm,amssymb, eucal}
\usepackage[backend=bibtex]{biblatex}
\usepackage{subfig}
\usepackage[tikz]{bclogo}
\usepackage{bbold}
\usepackage{xcolor}
\usepackage{stmaryrd} 
\usepackage{enumitem} 
\usepackage{lineno} 
\usepackage{hyperref} 
\hypersetup{
    colorlinks = true,
    linktoc = page
}
\usepackage{mathtools}
\mathtoolsset{showonlyrefs=true}
\usepackage{placeins}
\usepackage{xcolor}
\DeclareMathOperator*{\argmin}{argmin}

\begin{document}


\title{Linear-Quadratic McKean-Vlasov Stochastic Differential Games
\thanks{This work is  supported by FiME (Finance for Energy Market Research Centre) and the ``Finance et D\'eveloppement Durable - Approches Quantitatives'' EDF - CACIB Chair.}
}


\author{
Enzo MILLER
\footnote{LPSM, University Paris Diderot,  \sf enzo.miller at polytechnique.edu}
\qquad\quad
Huy\^en PHAM
\footnote{LPSM, University  Paris Diderot and CREST-ENSAE, \sf pham at lpsm.paris}
}


\maketitle


\begin{abstract} 
We consider a multi-player stochastic differential game with linear McKean-Vlasov dynamics and quadratic cost functional depending on the variance and mean of the state and control actions of the players in open-loop form. Finite and infinite horizon problems with possibly  some random coefficients as well as common noise are addressed.  We propose a simple direct approach based on weak martingale optimality principle together with a fixed point argument in the space of controls for solving this game problem. The Nash equilibria are   characterized in terms of  systems of Riccati ordinary differential equations and 
linear mean-field backward stochastic differential equations: existence and uniqueness conditions are provided for such systems. Finally, we illustrate our results on a toy example.
\end{abstract}

\vspace{5mm}


\vspace{5mm}


\newcommand{\vali}{\mathcal{V}\i}
\newcommand{\valspace}{\mathcal{C}_b^{1,2}([0,T ]\times \muspace[2])}

\newcommand{\cbstar}{\cbold^\star}
\newcommand{\crandspace}{L^0(\Omega, \mathcal{F}_t \lor \mathcal{G},\P;\boldsymbol{A})}
\newcommand{\actspace}{L^0(\Omega, \mathcal{F}_t \lor \mathcal{G},\P;\boldsymbol{A})}
\renewcommand{\c}[1][]{\alpha #1}
\newcommand{\cbold}[1][]{\boldsymbol{\alpha}#1}
\newcommand{\cvectT}{\VT{\c_1 - \bar{\c}_1}{\c_2 - \bar{\c}_2}}
\newcommand{\cvect}{\V{\c_1 - \bar{\c}_1}{\c_2 - \bar{\c}_2}}
\newcommand{\cspace}[1][]{
\ifthenelse{\equal{#1}{}}{ \boldsymbol{\mathcal{A}}}{\mathcal{A}^{#1}}
}
\newcommand{\s}{\sigma_{1,t}}
\newcommand{\st}{\tilde{\sigma}_{1,t}}
\renewcommand{\ss}{\sigma_{2,t}}
\newcommand{\sst}{\tilde{\sigma}_{2,t}}
\newcommand{\sx}{\sigma_{x,t}}
\newcommand{\sxt}{\tilde{\sigma}_{x,t}}

\renewcommand{\b}{b_{1,t}}
\newcommand{\bt}{\tilde{b}_{1,t}}
\newcommand{\bb}{b_{2,t}}
\newcommand{\bbt}{\tilde{b}_{2,t}}
\newcommand{\bx}{b_{x,t}}
\newcommand{\bxt}{\tilde{b}_{x,t}}

\newcommand{\X}[1][t]{X_#1}
\newcommand{\Xbar}{\bar{\X}}
\newcommand{\Xc}{(\X - \Xbar) }
\newcommand{\muspace}[1][]{
\ifthenelse{\equal{#1}{}}{\mathcal{P}_r(\R^m)}{\mathcal{P}_{#1}(\R^m)} }
\newcommand{\xispace}[1][]{
\ifthenelse{\equal{#1}{} }{L^r(\Omega,\mathcal{G},\P;\R^m)}{L^{#1}(\Omega,\mathcal{G},\P;\R^m)}
}

\newcommand{\scal}[2]{#1.#2^{\otimes2}}
\newcommand{\cond}[1]{\E[#1|W^0]}
\newcommand{\argscond}[1][t]{(#1, X^{\cbold}_{#1}, \cond{X^{\cbold}_{#1}}, \cbold[_#1], \cond{\cbold[_#1]})}
\newcommand{\args}[1][t]{(#1, X^{\cbold}_{#1}, \EE{X^{\cbold}_{#1}}, \cbold[_#1], \EE{\cbold[_#1]})}
\renewcommand{\arg}{(t,x,\bar{x},a_1, \bar{a_1}, a_2, \bar{a_2})}
\renewcommand{\bar}{\overline}
\newcommand{\inter}{\llbracket 1,n \rrbracket}
\renewcommand{\u}{\tilde{u}}
\newcommand{\du}{\mathrm{d}u}
\newcommand{\dd}{\mathrm{d}}
\newcommand{\proofspace}{\mbox{} \\*}
\newcommand{\tspace}{[0,T]}
\newcommand{\verteq}{\rotatebox{90}{$=$}}
\newcommand{\equalto}[2]{\underset{\overset{\verteq}{#2}}{#1}}
\newcommand{\aspace}[1][]{ 
\ifthenelse{ \equal{#1}{} }{ L^0(\mathcal{F}_t \lor \mathcal{G};\boldsymbol{\cali{A}}) } { L^0(\mathcal{F}_t \lor \mathcal{G};\cali{A}^i) }  }
\newcommand{\cali}[1]{ \mathcal{#1} }
\newcommand{\be}[1]{\begin{equation} #1 \end{equation}}
\newcommand{\bec}[1]{\begin{equation} \begin{cases} #1\end{cases} \end{equation}}
\newcommand{\bes}[1]{\begin{equation} \begin{split} #1\end{split} \end{equation}}
\newcommand{\B}[1]{\boldsymbol{#1}}
\newcommand{\1}{\mathbf{1}}
\newcommand{\var}{\mathrm{Var}}
\newcommand{\E}{\mathbb{E}}
\newcommand{\entrecro}[1]{\left[ #1 \right]}
\newcommand{\EE}[1]{\E \entrecro{#1}}
\newcommand{\entrebra}[1]{\left \{  #1 \right \}}
\newcommand{\entrepar}[1]{\left(  #1 \right) }
\newcommand{\R}{\mathbb{R}}
\renewcommand{\P}{\mathbb{P}}
\newcommand{\F}{\mathbb{F}}
\newcommand{\N}{\mathbb{N}}
\renewcommand{\i}[1][]{^{i #1}}
\newcommand{\xbar}{\bar{x}}
\newcommand{\V}[2]{
\begin{pmatrix} #1 \\ #2 \end{pmatrix}}
\newcommand{\VT}[2]{
\begin{pmatrix} #1, & #2 \end{pmatrix}}

\newtheorem{theorem}{Theorem}[section]
\newtheorem{hyp}{Hypothesis}
\newtheorem{corollary}[theorem]{Corollary}
\newtheorem{lemma}[theorem]{Lemma}
\newtheorem{defi}[theorem]{Definition}
\newtheorem{prop}[theorem]{Proposition}
\newtheorem{remark}[theorem]{Remark}
\newtheorem{csq}[theorem]{Consequence}


\section{Introduction} \label{secintro}

\subsection{General introduction-Motivation}

The study of large population of interacting individuals (agents, computers, firms) is a central issue in many fields of science, and finds numerous relevant applications in economics/finance (systemic risk with financial entities strongly interconnected), sociology (regulation of a crowd motion, herding behavior, social networks), physics,  biology, or electrical engineering (telecommunication). Rationality in the behavior of the population is a natural requirement, especially in social sciences, and is addressed by including individual decisions, where each individual optimizes some criterion, e.g. an investor maximizes her/his wealth, a firm  chooses  how much to produce outputs (goods, electricity, etc) or post advertising for a large population. The criterion and optimal decision of each individual depend on the others and affect the whole group, and one is then typically looking for an equilibrium among the population where the dynamics of the system evolves endogenously as a consequence of the optimal choices made by each individual. When the number of indistinguishable agents in the population tend to infinity, and by considering cooperation between the agents, we are reduced in the asymptotic formulation to a McKean-Vlasov (McKV)  control problem where the dynamics  and the cost functional depend upon the law of the stochastic process. This corresponds to a Pareto-optimum where a social planner/influencer decides of the strategies for each individual. 
The theory of  McKV  control problems, also called mean-field type control, has generated recent advances in the literature, either by the maximum principle \cite{CD15}, or the dynamic programming approach \cite{PW17}, see also  the recent books  \cite{BFY} and \cite{CD}, and the references therein, and linear quadratic (LQ) models provide an important class of solvable applications  studied in many papers, see, e.g., \cite{Yong}, \cite{HLY15},  \cite{Gr}, \cite{BP}.

In this paper, we consider multi-player stochastic differential games for McKean-Vlasov dyna\-mics. This corresponds and is motivated by the competitive  interaction of multi-population with a large number of indistinguishable agents.  In this context, we are then looking for a Nash equilibrium among the multi-class of populations.  Such problem, sometimes refereed to as mean-field-type game, allows to incorporate competition and heterogeneity in the population, and is a natural extension of McKean-Vlasov (or mean-field-type) control by including multiple decision makers. It  finds natural applications in engineering, power systems, social sciences and cybersecurity, and has  attracted recent attention in the literature, see, e.g., \cite{aurdje18}, \cite{cospha18}, \cite{djeetal18}, \cite{benetal18}.  
We focus more speci\-fically on the case of linear McKean-Vlasov dynamics and quadratic cost functional for each player (social planner). Linear Quadratic McKean-Vlasov stochastic differential game
has been studied in \cite{duntem18}  for a one-dimensional state process, and by restricting to closed-loop control. Here, we consider both finite and infinite horizon problems in a multi-dimensional framework, with random coefficients for the affine terms of the McKean-Vlasov dynamics and random coefficients for the linear terms of the cost functional. Moreover, controls of each player are in open-loop form. Our main contribution is to provide a simple and direct approach based on  weak martingale optimality principle developed in  \cite{BP} for McKean-Vlasov control problem, and that we extend  to the stochastic differential game, together with a fixed point argument in the space of open-loop controls, for finding a Nash equilibrium.  The key point is to find a suitable ansatz for determining the fixed point corresponding to the  Nash equilibria that we  characterize explicitly   in terms of  systems of Riccati ordinary differential equations and 
linear mean-field backward stochastic differential equations: existence and uniqueness conditions are provided for such systems. 

The rest of this paper is organized as follows. We continue Section \ref{secintro} by formulating the Nash equilibrium problem in the linear quadratic McKean-Vlasov finite horizon  framework, and by giving  
some notations and assumptions.  Section \ref{secweakmar} presents the verification lemma based on weak submartingale optimality principle for finding a Nash equilibrium, and details each step of the method to compute  a Nash equilibrium. We give some extensions in Section \ref{secexten} to the case of infinite horizon and common noise. Finally, we illustrate our results in Section \ref{secex} on some toy example.

\subsection{Problem formulation}

Let $T>0$ be a finite given horizon. Let $(\Omega, \mathcal{F},\mathbb{F} ,\mathbb{P})$ be a fixed filtered probability space where $\mathbb{F}=(\mathcal{F}_t)_{t\in[0,T]}$ is the natural filtration of a real Brownian motion $W$ $=$ $(W_t)_{t\in[0,T]}$.  In this section, for simplicity, we deal with the case of a single real-valued Brownian motion, 
and the case of multiple Brownian motions will be  addressed later in Section  \ref{secexten}.  We consider a multi-player game with $n$ players, and define the set of admissible controls for each player $i \in \inter $ as:

\begin{equation}
\mathcal{A}\i = \left\{ \c_i : \Omega \times [0,T] \to \R^{d_i}   \;  \text{s.t. \;} \c_i \text{\;is $\mathbb{F}$-adapted and} \int_0^T e^{-\rho t} \E[|\c_{i,t}|^2] dt < \infty   \right\},
\end{equation}
where $\rho$  is a nonnegative constant discount factor. We denote by $\cspace$ $=$ $\mathcal{A}^1 \times ... \times \mathcal{A}^n$, and for any  $\cbold = (\c_1, ..., \c_n) \in \cspace$, $i$ $\in$ $\inter$, we set $\alpha^{-i}$ $=$ $(\alpha_1,\ldots,\alpha_{i-1},\alpha_{i+1},\ldots,\alpha_n)$ $\in$ $\cspace[-i]$ $=$  $\cali{A}^1\times\ldots\times\cali{A}^{i-1}\times\cali{A}^{i+1}\times\ldots\times\cali{A}^n$.

\vspace{1mm}

Given a square integrable measurable random variable $X_0$ and control  $\cbold$ $=$ $(\c_1,...,\c_n)$ $\in$ $\cspace$,  we consider the controlled linear mean-field stochastic differential equation in $\R^d$:
\begin{equation}
\begin{cases}
d\X &= b \args dt + \sigma \args dW_t, \;\;\; 0 \leq t \leq T, \\
X_0^{\c} &= X_0,
\end{cases}
\label{dynamic}
\end{equation}
where for  $t \in [0,T]$, $x, \bar{x} \in \R^d$, $a_i, \bar{a}_i \in \R^{d_i}$: 
\begin{equation}
\begin{cases}
b(t,x,\bar{x}, \cbold, \bar{\cbold}) &= \beta_t + \bx x + \tilde{b}_{x,t}\bar{x} + \sum_{i=1}^n b_{i,t} \c_{i} + \tilde{b}_{i,t} \bar{\c}_{i}  \\
&= \beta_t + \bx x + \tilde{b}_{x,t}\bar{x} + B_t \cbold + \tilde{B}_t \bar{\cbold} \\
\sigma (t,x,\bar{x}, \cbold, \bar{\cbold})  &= \gamma_t + \sx x + \tilde{\sigma}_{x,t}\bar{x} + \sum_{i=1}^n \sigma_{i,t} \c_{i} + \tilde{\sigma}_{i,t} \bar{\c}_{i} \\
&= \gamma_t + \sx x + \tilde{\sigma}_{x,t}\bar{x} + \Sigma_t \cbold + \tilde{\Sigma}_t \bar{\cbold}.
\end{cases}
\label{coefX}
\end{equation}
Here all the coefficients are deterministic matrix-valued processes except $\beta$ and $\sigma$ which are vector-valued $\mathbb{F}$-progressively measurable processes.

The goal of each player $i \in \inter $ during the game is to minimize her cost functional over $\alpha_i$ $\in$ 
$\mathcal{A}\i$, given the actions $\alpha^{-i}$ of the other players:  
\begin{eqnarray*}
J\i(\alpha_i,\alpha^{-i}) & = &  \E \Big[ \int_0^T  e^{-\rho t} f\i \args  dt + g\i(X_T^{\cbold},\E[X_T^{\cbold}] )\Big], \\
V^i(\alpha^{-i}) & = & \inf_{\alpha_i \in \mathcal{A}\i} J\i(\alpha_i,\alpha^{-i}), 
\end{eqnarray*}
where for each $t \in [0,T]$, $x, \bar{x} \in \R^d$, $a_i, \bar{a}_i \in \R^{d_i}$,  we have set the running cost and terminal cost for each player:
\begin{equation}
\begin{cases}
f\i (t,x,\bar{x}, \boldsymbol{a}, \boldsymbol{\bar{a}}) &= (x-\xbar )^\intercal Q\i _t (x-\xbar ) + \xbar ^\intercal[Q\i _t + \tilde{Q}\i _t] \xbar  \\
& + \sum_{k=1}^n a_k^\intercal I_{k,t}\i (x-\bar{x}) + \bar{a}_k^\intercal (I_{k,t}\i + \tilde{I}_{k,t}\i ) \bar{x}  \\
& + \sum_{k=1}^n (a_k - \bar{a}_k)^\intercal N_{k,t}\i (a_k - \bar{a}_k) + \bar{a}_k (N_{k,t}\i + \tilde{N}_{k,t}\i ) \bar{a}_k \\
& +  \sum_{0 \le k \ne l \le n} (a_k - \bar{a}_k)^\intercal G\i_{k,l,t} (a_l - \bar{a}_l) +  a_k^\intercal ( G\i_{k,l,t} + \tilde{G}_{k,l,t}\i  )a_l\\
& + 2[L_{x,t}\i[T] x + \sum_{k=1}^n L_{k,t}^{i\intercal} a_k] \\
g\i (x, \xbar ) &= (x-\xbar )^\intercal P\i  (x-\xbar ) + \xbar (P\i  + \tilde{P}\i )\xbar  + 2 r^{i\intercal} x. 
\end{cases}
\label{coefJ}
\end{equation}
Here all the coefficients are deterministic matrix-valued processes, except $L_x\i, L_k\i, r\i$ which are vector-valued $\mathbb{F}$-progressively measurable processes, and $\intercal$ denotes the transpose of a vector or matrix.

\vspace{2mm}

We say that $\cbold^* = (\c_1^*, ..., \c_n^*) \in \cspace$ is a Nash equilibrium if for any $i \in \inter$, 
\begin{equation}
J^i(\cbold^*) \; \leq \; J^i(\c_i,\alpha^{*,-i}), \;\;\; \forall \alpha_i \in  \cspace[i],\mbox{ i.e. },  \;  J^i(\cbold^*) \; = \; V^i(\alpha^{*,-i}). 
\end{equation}
As it is well-known, the search for a Nash equilibrium can be formulated as a fixed point problem as follows: first, each player $i$ has to compute its best response given the controls of the other players: 
$\c^\star_i = BR_i(\c^{-i})$, where $BR_i$ is the best response function defined (when it exists) as:
\begin{align*}
  BR_i \colon \cspace[-i] &\to \cspace[i]\\
  \c^{-i} &\mapsto \argmin_{\c \in \cspace[i]} J\i(\c, \c^{-i}). 
\end{align*}
Then, in order to ensure that $(\c^\star_1, ..., \c^\star_n)$ is a Nash equilibrium, we have to check that this candidate verifies the fixed point equation: 
$(\c^\star_1, ..., \c^\star_i) = BR(\c^\star_1, ..., \c^\star_i)$ where $BR := (BR_1, ...BR_n)$.

The main goal  of this paper is to state a general martingale optimality principle for the search of Nash equilibria and to apply it to the linear quadratic case. We first obtain best response functions (or optimal control of each agent conditioned to the control of the others) of each player $i$ of the following form: 
\begin{equation}
\begin{split}
\c_{i,t} &=- ({S}^{i }_{i,t})^{-1} {U}^{i }_{i,t} (X_t-\E[X_t]) - ({S}^{i }_{i,t})^{-1} ({\xi}\i_{i,t} - {\bar{\xi}}\i_{i,t}) -(\hat{S}_{i,t}^{i })^{-1}(V^{i}_{i,t} \E[X_t] + O^{i}_{i,t})
\end{split}
\end{equation}
where the coefficients in the r.h.s., defined in \eqref{coeffX_a} and \eqref{coeffA_a}, depend on the actions $\alpha^{-i}$ of the other players. 
We then proceed  to a fixed point search for best response function in order to exhibit a Nash equilibrium that is described in Theorem \ref{theorem_nash_eq}.

\subsection{Notations and Assumptions}


Given a normed space  $(\mathbb{K},|.|)$,  and for  $T \in \R^\star_+$, we set: 
\begin{equation*}
\begin{split}
L^{\infty}([0,T], \mathbb{K}) &= \left\{ \phi : [0,T] \to \mathbb{K} \text{\;s.t. $\phi$ is measurable and} \sup_{t\in [0,T]} |\phi_t| < \infty  \right\} 
\\
L^{2}([0,T], \mathbb{K}) &= \left\{ \phi : [0,T] \to \mathbb{K} \text{\;s.t. $\phi$ is measurable and} \int_0^T  e^{-\rho u}|\phi_t|^2  du  < \infty  \right\} 
\\
L^{2}_{\mathcal{F}_T}(\mathbb{K}) &= \left\{ \phi : \Omega \to \mathbb{K} \text{\;s.t. $\phi$ is $\mathcal{F}_T$-measurable and\;} \E[|\phi|^2] < \infty  \right\} 
\\
\mathcal{S}_{\mathbb{F}}^{2}(\Omega \times [0,T], \mathbb{K}) &= \left\{ \phi : \Omega \times [0,T] \to \mathbb{K} \text{\;s.t. $\phi$ is $\mathbb{F}$-adapted and\;} \E[\sup_{t\in [0,T]}   |\phi_t|^2] < \infty  \right\} 
\\
L^{2}_{\mathbb{F}}( \Omega \times [0,T], \mathbb{K}) &= \left\{ \phi : \Omega \times [0,T] \to \mathbb{K} \text{\;s.t. $\phi$ is $\mathbb{F}$-adapted and} \int_0^T e^{-\rho u} \E[|\phi_u|^2]  du < \infty  \right\}.  
\end{split}
\end{equation*}
Note that when we will tackle the infinite horizon case we will set $T = \infty$. To make the notations less cluttered,  we sometimes denote $X$ $=$ $X^{\cbold}$ when there is no ambiguity. 
If $C$ and $\tilde{C}$ are coefficients of our model, either in the dynamics  or in a cost function, we note: $\hat{C} = C + \tilde{C}$. Given  a random variable $Z$  with a first moment, we denote by  $\bar{Z} = \E[Z]$. For  $M \in \R^{n\times n}$ and $X\in \R^n$, we denote by  $\scal{M}{X} = X^\intercal M X \in \R$. We  denote by $\mathbb{S}^d$ the set of symmetric  $d\times d$ matrices and by $\mathbb{S}^d_+$ the subset of non-negative symmetric matrices.

\vspace{2mm}

Let us now  detail here the assumptions on the coefficients. 

\vspace{2mm}

\textbf{(H1)} The coefficients in the dynamics \eqref{coefX} satisfy:

\begin{enumerate}[label= \alph*),leftmargin=2cm ,parsep=0cm,itemsep=0cm,topsep=0cm]
    \item $\beta, \gamma \in L^{2}_{\mathbb{F}}( \Omega \times [0,T], \R^d)  $
    \item $b_x, \tilde{b}_x, \sigma_x, \tilde{\sigma}_x  \in L^{\infty}([0,T], \R^{d\times d}) $; $b_i, \tilde{b}_i, \sigma_i, \tilde{\sigma}_i \in L^{\infty}([0,T], \R^{d\times d_i})$
\label{H1}
\end{enumerate}

\vspace{2mm}

\textbf{(H2)} The coefficients of the cost functional \eqref{coefJ} satisfy:
\begin{enumerate}[label= \alph*),leftmargin=2cm ,parsep=0cm,itemsep=0cm,topsep=0cm]
  \item $Q\i, \tilde{Q}\i \in L^{\infty}([0,T], \mathbb{S}^d_+)$, $P\i,\tilde{P}\i \in \mathbb{S}^d$, $N\i_k, \tilde{N}_k\i \in L^{\infty}([0,T], \mathbb{S}^{d_k}_+)$, $I\i_k, \tilde{I}_k\i \in L^{\infty}([0,T], \R^{d_k \times d})$
  \item $L_x\i \in L^{2}_{\mathbb{F}}( \Omega \times [0,T], \R^d)$, $L_k\i \in L^{2}_{\mathbb{F}}( \Omega \times [0,T], \R^{d_k})$, $r\i \in L^{2}_{\mathcal{F}_T}(\R^d)$
  \item \label{H2_c} $\exists \delta >0 \; \forall t \in [0,T]$:\\
  $N\i_{i,t} \geq \delta \mathbb{I}_{d_k} \;\;\;\;\; P\i \geq 0 \;\;\;\;\; Q\i_t - I_{i,t}^{i\intercal}(N_{i,t}^{i})^{-1} I_{i,t}^{i} \geq 0 $
   \item \label{H2_d} $\exists \delta >0 \; \forall t \in [0,T]$:\\
  $\hat{N}\i_{i,t} \geq \delta \mathbb{I}_{d_k} \;\;\;\;\; \hat{P}\i \geq 0 \;\;\;\;\; \hat{Q}\i_t - \hat{I}_{i,t}^{i\intercal}(\hat{N}_{i,t}^{i})^{-1}  \hat{I}_{i,t}^{i} \geq 0.$
\end{enumerate}

 \vspace{3mm}

 Under the above conditions, we easily derive some standard estimates on the mean-field SDE:

 \vspace{2mm}

\begin{enumerate}[label= -,leftmargin=0cm ,parsep=0cm,itemsep=0cm,topsep=0cm]
\item By \textbf{(H1)} there exists a unique strong solution to the mean-field SDE \eqref{dynamic}, which verifies:
\begin{equation}
\E \Big[\sup_{t \in [0,T]} |X_t^{\c}|^2 \Big] \leq C_{\c}(1 + \E(|X_0|^2)) < \infty
\label{estimateX}
\end{equation}
where $C_{\c}$ is a constant which depending on $\c$ only through $\int_0^T  e^{-\rho t} \E[|\c_t|^2]dt.$
\item By \textbf{(H2)} and \eqref{estimateX} we have:
\begin{equation}
J\i(\cbold) \in \R \text{ for each $\cbold \in \cspace$}, 
\end{equation}
which means that the optimisation problem is well defined for each player.

\end{enumerate}


\section{A Weak submartingale optimality principle to compute a Nash-equilibrium} \label{secweakmar}

\subsection{A verification Lemma}

We first present the lemma on which the method is based.

\begin{lemma}[Weak submartingale optimality principle]
\label{optimPrinciple}
Suppose there exists a couple \\ $(\cbold^\star , (\cali{W}^{.,i})_{i \in \inter })$,
where $\cbold^\star \in \cspace$ and  $\cali{W}^{.,i}=\{ \cali{W}^{\cbold,i}_t, t \in [0,T], \cbold \in \cspace  \}$ is a family of adapted processes indexed by $\cspace$ for each $i\in \inter$, such that:

\begin{enumerate}[label=\roman*]
\item[(i)] For every $\cbold \in \cspace$, $\E[\cali{W}^{\cbold,i}_0]$ is independent of the control $\c_i \in \cspace[i]$; \label{item1}
\item[(ii)] For every $\cbold \in \cspace$, $\E[\cali{W}_T^{\cbold,i}] = \E[g\i(X_T^{\cbold}, \P_{X_T^{\cbold}})]$; \label{item2}
\item[(iii)] For every $\cbold \in \cspace$, the map $t \in [0,T] \mapsto \E[\cali{S}_t^{\cbold,i}]$, with \\ $\cali{S}_t^{\cbold,i} = e^{-\rho t } \cali{W}^{\cbold,i}_t + \int_0^t e^{-\rho u } f\i(u, X_u^{\cbold}, \P_{X_u^{\cbold}}, \cbold_u, \P_{\cbold_u}) du $ is well defined and non-decreasing; \label{item3}
\item[(iv)]  The map $t \mapsto \E[\cali{S}_t^{\cbstar,i}]$ is constant for every $t \in [0,T]$; \label{item4}
\end{enumerate}
Then $\cbold^\star$ is a Nash equilibrium and $J\i(\cbstar) = \E[\cali{W}_0^{\cbstar,i}]$. Moreover, any other Nash-equilibrium $\tilde{\cbold}$ such that $\E[\cali{W}_0^{\tilde{\cbold},i}] = \E[\cali{W}_0^{{\cbold^\star},i}]$ and $J\i(\tilde{\cbold}) = J\i({\cbold^\star})$ for any $i \in \inter$ satisfies the condition (iv).
\label{lemma_verification}
\end{lemma}
\begin{proof}
Let $i \in \inter$ and $\c_i \in \cspace[i]$. From (ii)  we have immediately $J\i(\cbold) = \EE{\cali{S}_T^{\cbold}}$ for any $\cbold \in \cspace$. We then have:
\begin{equation}
\begin{split}
\E[\cali{W}^{(\c_i, \cbold{\star,-i}),i}_0] &= \E[\cali{S}^{(\c_i, \cbold^{\star,-i}),i}_0] \\
& \leq \E[\cali{S}^{(\c_i, \cbold^{\star-i}),i}_T] \; = \;  J\i(\c_i, \cbold^{\star,-i}). 
\end{split}
\end{equation}
Moreover for $\c_i = \c_i^\star$ we have: 
\begin{equation*}
\begin{split}
\E[\cali{W}^{(\c_i^\star, \cbold^{\star,-i}),i}_0] &= \E[\cali{S}^{(\c_i^\star, \cbold^{\star,-i}),i}_0] \\
& = \E[\cali{S}^{(\c_i, \cbold^{\star,-i}),i}_T] \;  =  \; J\i(\c_i^\star, \cbold^{\star,-i}), 
\end{split}
\end{equation*}
which proves that $\cbstar$ is a Nash equilibrium and $J\i(\cbstar) = \E[\cali{W}_0^{\cbstar,i}]$. Finally, let us suppose that $\tilde{\cbold} \in \cspace$ is another Nash equilibrium such that $\E[\cali{W}_0^{\tilde{\cbold},i}] = \E[\cali{W}_0^{{\cbold^\star},i}]$ and $J\i(\tilde{\cbold}) = J\i({\cbold^\star})$ for any $i \in \inter$. Then, for $i \in \inter$ we have:
\begin{equation*}
\E[\cali{S}^{\tilde{\cbold},i}_0] =  \E[\cali{W}_0^{\tilde{\cbold},i}] = \E[\cali{W}_0^{{\cbold}^\star,i}] = \E[\cali{S}^{{\cbold}^\star,i}_T] = J\i(\cbold^\star) = J\i(\tilde{\cbold})=\E[\cali{S}^{\tilde{\cbold},i}_T]. 
\end{equation*}
Since $t \mapsto \E[\cali{S}_t^{\tilde{\cbold},i}]$ is nondecreasing for every $i \in \inter$, this implies that the map is actually constant and (iv)  is verified.
\end{proof}



\subsection{The method and the solution}
Let us now apply the optimality principle in Lemma \ref{optimPrinciple} in order to find a Nash equilibrium. In the linear-quadratic case the laws of the state and the controls intervene only through their expectations. Thus we will use a simplified optimality principle where $\P$ is simply replaced by $\E$ in  
conditions  (ii)  and (iii)   of Lemma \ref{optimPrinciple}. 
The general procedure is the following:
 
\begin{itemize}[label={}]
\item \textbf{Step 1.} We guess a candidate for $\cali{W}^{\cbold,i}$. To do so we suppose that $\cali{W}^{\cbold,i}_t = w_t\i(X_t^{\cbold}, \EE{X_t^{\cbold}})$  for some parametric adapted random field $\{w_t\i(x,\bar{x}),t\in [0,T], x,\bar{x} \in \R^d \}$ of the form $w_t\i(x,\bar{x}) =  \scal{K_t\i}{(x-\bar{x})} + \scal{\Lambda_t\i }{\bar{x}} +2 Y_t^{i \intercal}x + R\i_t$.
\item \textbf{Step 2.} We set $\cali{S}^{\cbold,i}_t =  e^{-\rho t } w_t\i(X_t^{\cbold}, \EE{X_t^{\cbold}}) + \int_0^t e^{-\rho u } f\i(u,X_u^{\cbold}, \EE{X_u^{\cbold}}, \cbold_u, \EE{\cbold_u}) du  $ for $i \in \inter$ and $\cbold \in \cspace$ .We then compute $\frac{d}{dt} \E \big[ {\cali{S}_t^{\cbold,i}} \big]  = e^{-\rho t}\E \big[{D^{\c,i}_t} \big]$ (with It\^o's formula) where the drift $D^{\cbold,i}$ takes the form:
\begin{equation*}
\E \big[ {D^{\c,i}_t} \big]  \; = \;  \E \Big[ {-\rho w_t\i(X_t^{\cbold}, \EE{X_t^{\cbold}}) + \frac{d}{dt} \EE{w_t\i(X_t^{\c},\EE{X_t^{\c}} ) } + f\i \args[t]} \Big].
\end{equation*} 
\item \textbf{Step 3.} We then constrain the coefficients of the random field so that the conditions of  Lemma \ref{optimPrinciple} are satisfied. This leads to a system of backward ordinary and stochastic differential equations for the coefficients of $w\i$.
\item \textbf{Step 4.} At time $t$, given the state and the controls of the other players, we seek the action $\alpha_i$ cancelling the drift. We thus obtain the best response function of each player. 
\item \textbf{Step 5.} We compute the fixed point of the best response functions in order to find an open loop Nash equilibrium $t \mapsto \cbstar_t$. 
\item \textbf{Step 6.} We check the validity of our computations.
\end{itemize}

\subsubsection{Step 1:  guess the random fields}

The process $t \mapsto \EE{w_t\i(X_t^{\cbold}, \EE{X_t^{\cbold}})}$ is meant to be equal to $\EE{g\i(X_T^{\cbold}, \EE{X_T^{\cbold}})}$ at time $T$, where 
$g(x,\bar{x})=\scal{P\i}{(x-\bar{x})}+ \scal{(P\i + \tilde{P}\i)}{\bar{x}} + r^{i\intercal}x$ with 
$(P,\tilde{P}, r\i) \in (\mathbb{S}^d)^2 \times L^{2}_{\mathcal{F}_T}(\R^d)$. It is then natural to search for a field $w\i$ of the form $w_t\i(x,\bar{x}) = \scal{K_t\i}{(x-\bar{x})} + \scal{\Lambda_t\i }{\bar{x}} +2 Y_t^{i\intercal}x + R_t\i$ with the processes $(K\i,\Lambda\i, Y\i, R\i)$  in $(L^{\infty}([0,T], \mathbb{S}^d_+)^2\times \cali{S}^2_{\mathbb{F}}(\Omega\times [0,T],\R^{d}) \times L^{\infty}([0,T], \R)$ and solution to:
\begin{equation}
\begin{cases}
dK\i_t = \dot{K}_t\i dt, &K\i_T = P\i \\
d\Lambda\i_t = \dot{\Lambda}_t\i dt, &\Lambda\i_T = P\i + \tilde{P}\i \\
dY_t\i = \dot{Y}_t\i dt + Z_t\i dW_t, \;\; 0 \leq t \leq T,  &Y_T\i = r\i \\
dR_t\i = \dot{R}_t\i dt,  &R\i_T = 0, 
\end{cases}
\end{equation}
where $(\dot{K}\i, \dot{\Lambda}\i, \dot{R}\i)$ are deterministic processes valued in $\mathbb{S}^d \times \mathbb{S}^d \times \R$ and $(\dot{Y}\i, Z\i)$ are adapted processes valued in $\R^d$.

\subsubsection{Step 2:  derive  their drifts }
For $i \in \inter$, $t \in [0,T]$ and $\cbold \in \cspace$, we set: 
\begin{equation}
\mathcal{S}_t^{\cbold, i} := e^{-\rho t} w\i_t(\X, \EE{\X}) + \int_0^t e^{-\rho u} f\i \args[u] \du
\end{equation}
and then compute the drift of  the deterministic function $t \mapsto \E[\cali{S}^{\cbold,i}_t]$:
\begin{equation*}
\begin{split}
\frac{d \E[\mathcal{S}_t^{\cbold, i} ]}{dt} &=  e^{-\rho t} \E [D^{\cbold, i}_t] \\
&= e^{-\rho t} \E[ \Xc^\intercal [\dot{K}\i_t + \Phi \i_t] \Xc + \Xbar^\intercal(\dot{\Lambda}\i_t + \Psi\i_t) \Xbar +2 [\dot{Y}\i_t + \Delta\i_t]^\intercal \X \\
&  +\dot{R}\i_t -\rho R_t^{i} + \bar{\Gamma}\i_t + \chi \i _t (\c_{i,t}) ],
\end{split}
\end{equation*}
where we have defined:
\begin{equation*}
\begin{split}
\chi \i _t (\c_{i,t}) &:=(\c[_{i,t}] - \bar{\c}_{i,t})^\intercal S\i_{i,t} (\c[_{i}] - \bar{\c}_{i,t}) + \bar{\c}_{i,t}^\intercal \hat{S}\i_{i,t} \bar{\c}_{i,t} +2 [U^{i}_{i,t}(\X-\bar{\X}) + V^{i}_{i,t} \bar{\X} + O^{i}_{i,t} + \xi^{i}_{i,t} - \bar{\xi}^{i}_{i,t}]^\intercal \c[_{i,t}] 
\end{split}
\end{equation*}
with the following coefficients:
\begin{equation}
\begin{cases}
\Phi \i_t &= Q\i _t + \sx^\intercal K\i_t \sx + K\i_t \bx + \bx^\intercal K\i_t - \rho K\i_t   \\
\Psi \i_t &= \hat{Q}\i _t + \hat{\sigma}_{x,t}^\intercal K\i_t \hat{\sigma}_{x,t} + \Lambda\i_t \hat{b}_{x,t} + \hat{b}_{x,t}^\intercal \Lambda\i_t - \rho \Lambda\i_t \\
\Delta \i_t & = L\i_{x,t} + \bx^\intercal Y\i_t + \bxt^\intercal \bar{Y}\i_t + \sx^\intercal Z\i_t + \sxt^\intercal \bar{Z}\i_t + \Lambda\i_t \bar{\beta}_t \\
&+ \sx^\intercal K\i_t \gamma_t + \sxt^\intercal K\i_t \bar{\gamma}_t + K\i_t(\beta_t - \bar{\beta}_t)
- \rho Y\i_t\\
&+ \sum_{k \neq i} U^{i\intercal}_{k,t} (\c_{k,t} - \bar{\c}_{k,t}) + V^{i\intercal}_{k,t} \bar{\c}_{k,t} \\
\Gamma\i_t &= \gamma_t^\intercal K\i_t \gamma_t + 2 \beta_t^\intercal Y\i_t + 2 \gamma_t^\intercal Z\i_t \\
&+ \sum_{k \neq i} (\c[_{k,t}] - \bar{\c}_{k,t})^\intercal S\i_{k,t} (\c[_{k,t}] - \bar{\c}_{k,t}) + \bar{\c}_{k,t}^\intercal \hat{S}\i_{k,t} \bar{\c}_{k,t} +2 [O^{i}_{k,t} + \xi^{i}_{k,t} - \bar{\xi}^{i}_{k,t}]^\intercal \c[_{k,t}] - \rho R\i_t,
\label{coeffX_a}
\end{cases}
\end{equation}
and 
\begin{equation}
\begin{cases}
S\i_{k,t} &= N\i_{k,t} + \sigma_{k,t}^\intercal K\i_t \sigma_{k,t}  \\
\hat{S}\i_{k,t} &= \hat{N}\i_{k,t} + \hat{\sigma}_{k,t}^\intercal K\i_t  \hat{\sigma}_{k,t} \\
U^{i}_{k,t} &= I_{k,t}\i + \sigma_{k,t}^\intercal K\i_t \sigma_{x,t} + b_{k,t}^\intercal K\i_t \\
V^{i}_{k,t} &= \hat{I}_{k,t}\i + \hat{\sigma}_{k,t}^\intercal  K\i_t \hat{\sigma}_{x,t} + \hat{b}_{k,t}^\intercal \Lambda \i_t \\
O^{i}_{k,t} &= \bar{L}\i_{k,t} + \hat{b}_{k,t}^\intercal \bar{Y}\i_t + \hat{\sigma}_{k,t}^\intercal \bar{Z}\i_t + \hat{\sigma}_{k,t}^\intercal K\i_t \bar{\gamma}_t \\ 
&+ \frac{1}{2} \sum_{k \neq i} (\hat{J}\i_{i,k,t} + \hat{J}_{k,i,t}^{i\intercal})\bar{\c}_{k,t} \\
J\i_{k,l,t} &= G\i_{k,l,t} +  \sigma_{k,t}^\intercal K\i_t \sigma_{l,t} \\
\hat{J}\i_{k,l,t} &= \hat{G}\i_{k,l,t} +  \hat{\sigma}_{k,t}^\intercal K\i_t \hat{\sigma}_{l,t} \\
\xi^{i}_{k,t} &= L\i_{k,t} + b_{k,t}^\intercal Y\i_t + \sigma_{k,t}^\intercal Z\i_t + \sigma_{k,t}^\intercal K\i_t \gamma_t \\
&+ \frac{1}{2} \sum_{k \neq i} (J\i_{i,k,t} + J_{k,i,t}^{i\intercal})\c_{k,t}. 
\label{coeffA_a}
\end{cases}
\end{equation}

\subsubsection{Step 3:  constrain their coefficients}
Now that we have computed the drift, we need to constrain the coefficients so that $\cali{S}^{\cbold,i}$ satisfies the condition of Lemma \ref{optimPrinciple}. Let us assume for the moment that $S\i_{i,t}$ and $\hat{S}\i_{i,t}$ are positive definite matrices (this will be ensured  by the positive definiteness of $K$). That implies that there exists an invertible matrix $\theta\i_t$ such that $\theta\i_t S\i_{i,t} \theta_t^{i \intercal}= \hat{S}\i_{i,t}$ for all $t \in [0,T]$. We can now rewrite the drift as: \textit{"a square in $\c_i$" + "other terms not depending in $\c_i$}".
Since we can form the following square:
\begin{equation}
\E[\chi \i _t (\c_{i,t})] = \E[ (\c_{i,t} - \bar{\c}_{i,t} + \theta^{i \intercal}_t \bar{\c}_{i,t} - \eta \i _t ) S\i_{i,t} (\c_{i,t} - \bar{\c}_{i,t} + \theta^{i \intercal}_t \bar{\c}_{i,t} - \eta \i _t ) - \zeta_t\i ]
\end{equation}
with: 
\begin{equation}
\begin{cases}
\eta \i _t &= a^{i,0}_t(\X, \Xbar) + \theta^{i \intercal}_t a^{i,1}_t(\Xbar) \\
a^{i,0}_t(x, \bar{x})  &= - \big({S}^{i }_{i,t}\big)^{-1} {U}^{i }_{i,t} (x-\bar{x})
                    - \big({S}^{i }_{i,t}\big)^{-1} ({\xi}\i_{i,t} - {\bar{\xi}}\i_{i,t}) \\
a^{i,1}_t(\bar{x}) &= - \big(\hat{S}_{i,t}^{i }\big)^{-1}(V^{i}_{i,t} \bar{x} + O^{i}_{i,t}) \\
\zeta\i_t &= (\X - \Xbar)^\intercal  U_{i,t}^{i \intercal} \big(S^{i }_t\big)^{-1} U_t\i (\X - \Xbar) + \Xbar {V_{i,t}^{i \intercal} \big(\hat{S}_t^{i }\big)^{-1}V_{i,t}^{i}}  \Xbar \\
& + 2(U_{i,t}^{i \intercal} \big(S_t^{i}\big)^{-1}  (\xi\i_{i,t} - \bar{\xi}\i_{i,t})  + 
V_{i,t}\i \big(\hat{S}^{i}_{i,t} \big)^{-1} O\i_{i,t}) \X \\
& + (\xi\i_{i,t} - \bar{\xi}\i_{i,t})^\intercal \big(S_{i,t}^{i}\big)^{-1}(\xi\i_{i,t} - \bar{\xi}\i_{i,t}) 
+ O_{i,t}^{i\intercal} \big(\hat{S}^{i}_{i,t}\big)^{-1} O_{i,t}\i, 
\end{cases}
\end{equation}
we can then rewrite the drift in the following form:
\begin{equation}
\begin{split}
\E [D^{\cbold, i}_t] &= \E[ \Xc^\intercal [\dot{K}\i_t + \Phi^{i0}_t] \Xc + \Xbar^\intercal(\dot{\Lambda}\i_t + \Psi^{i0}_t) \Xbar +2 [\dot{Y}\i_t + \Delta^{i0}_t]^\intercal \X \\
& + \dot{R}\i_t + \bar{\Gamma_t^{i0}}  \\
& + (\c_{i,t} - \bar{\c}_{i,t} + \theta^{i \intercal}_t \bar{\c}_{i,t} - \eta \i _t ) S\i_{i,t} (\c_{i,t} - \bar{\c}_{i,t} + \theta^{i \intercal}_t \bar{\c}_{i,t} - \eta \i _t )
],
\end{split}
\end{equation}
where 
\begin{equation}
\begin{cases}
\Phi^{i0}_t &= \Phi^{i}_t - {U_{i,t}^{i \intercal} \big(S^{i }_{i,t}\big)^{-1} U_{i,t}^{i }} \\
\Psi^{i0}_t &= \Psi^{i}_t - {V_{i,t}^{i \intercal} \big(\hat{S}_{i,t}^{i }\big)^{-1} V_{i,t}^{i }} \\
\Delta^{i0}_t &= \Delta^{i}_t - U_{i,t}^{i \intercal} \big(S_{i,t}^{i}\big)^{-1}  (\xi\i_{i,t} - \bar{\xi}\i_{i,t})  -  V_{i,t}^{i \intercal} \big(\hat{S}^{i }_t\big)^{-1} O\i_{i,t} \\
\Gamma_t^{i0} &= \Gamma_t^{i} - (\xi\i_{i,t} - \bar{\xi}\i_{i,t})^\intercal \big(S_{i,t}^{i} \big)^{-1}(\xi\i_{i,t} - \bar{\xi}\i_t) - O_{i,t}^{i\intercal} \big(\hat{S}_{i,t}^{i}\big)^{-1} O_{i,t}\i. 
\end{cases}
\label{coeff_0}
\end{equation}

We can finally constrain the coefficients. By choosing the coefficients $K\i, \Gamma\i, Y\i$ and $R\i$  so that only the square remains, the drift for each  player $i \in \llbracket 1,n \rrbracket$ can be rewritten as a square only (in the next step we will verify that we can indeed choose such coefficients). More precisely we set $K\i, \Gamma\i, Y\i$ and $R\i$ as the solution of:
\begin{equation}
\begin{cases}
dK\i_t = -\Phi^{i0}_t dt &K\i_T = P\i \\
d\Lambda\i_t = -\Psi^{i0}_t dt &\Lambda\i_T = P\i + \tilde{P}\i \\
dY_t\i = -\Delta^{i0}_t dt + Z_t\i dW_t &Y_T\i = r\i \\
dR_t\i = -\Gamma_t^{i0} dt &R\i_T = 0, 
\end{cases}
\label{sys_coeff_optimControl}
\end{equation}
and stress the fact that $Y\i, Z^i, R\i$ depend on $\alpha^{-i}$, which appears in the coefficients  $\Delta^{i0}$, and $\Gamma^{i0}$.  
With such coefficients the drift takes now the form:
\begin{equation}
\begin{split}
\E [D^{\cbold, i}_t] &= \E[ (\c_{i,t} - \bar{\c}_{i,t} + \theta^{i \intercal}_t \bar{\c}_{i,t} - \eta \i _t ) S\i_{i,t} (\c_{i,t} - \bar{\c}_{i,t} + \theta^{i \intercal}_t \bar{\c}_{i,t} - \eta \i _t )
] \\
&=\E [ (\c_{i,t} - \bar{\c}_{i,t} - a^{i,1}_t+ \theta^{i \intercal}_t( \bar{\c}_{i,t} - a^{i,0}_t))  S\i_{i,t}  (\c_{i,t} - \bar{\c}_{i,t} - a^{i,1}_t+ \theta^{i \intercal}_t( \bar{\c}_{i,t} - a^{i,0}_t)) ]
\end{split}
\label{drift_optimal_control}
\end{equation}
and thus satisfies the nonnegativity constraint:  $\E [D^{\cbold, i}_t]$ $\geq$ $0$, for all $t$ $\in$ $[0,T]$, $i$ $\in$ $\inter$, and $\cbold$ $\in$  $\cspace$.

\subsubsection{Step 4:  find the best response functions}

\begin{prop}
Assume that for all $i \in \inter$,  $(K\i,\Lambda\i, Y\i, Z\i, R\i)$ is a solution  of \eqref{sys_coeff_optimControl} given $\alpha^{-i}$ $\in$ ${\cal A}^{-i}$. 
Then the set of processes 
\begin{equation}
\begin{split}
\c_{i,t} &= a_t^{i,0}(X_t, \E[X_t]) + a_t^{i,1}(\E[X_t]) \\
&=- \big({S}^{i }_{i,t}\big)^{-1}  {U}^{i }_{i,t} (X_t-\E[X_t]) - \big({S}^{i }_{i,t}\big)^{-1} ({\xi}\i_{i,t} - {\bar{\xi}}\i_{i,t}) - \big(\hat{S}_{i,t}^{i }\big)^{-1}(V^{i}_{i,t} \E[X_t] + O^{i}_{i,t})
\end{split}
\label{best_response}
\end{equation}
(depending on $\alpha^{-i}$) where $X$ is the state process with the feedback controls $\cbold = (\c_1,...,\c_n)$, are best-response functions, i.e., 
$J^i(\c_i,\alpha^{-i})$ $=$ $V^i(\alpha^{-i})$ for all $i \in \inter$.  
Moreover we have 
\begin{equation*}
\begin{split}
V^i(\alpha^{-i})
&= \E[\cali{W}^{i,\cbold}_0]\\
&=\E[\scal{K_0\i}{(X_0-\bar{X}_0)} + \scal{\Lambda_0\i }{\bar{X}_0} +2 Y_0^{i\intercal}X_0 + R_0\i]. 
\end{split}
\end{equation*}
\label{prop_nash_eq}
\end{prop}
\begin{proof}
We check that the assumptions of Lemma \ref{lemma_verification} are satisfied. Since $\cali{W}^{\cbold ,i}$ is of the form $\cali{W}^{\cbold,i}_t = w_t\i(X_t^{\cbold}, \E[X_t^{\cbold}])$, condition (i)  is verified. The condition (ii)  is satisfied thanks to the terminal conditions imposed on the system \eqref{sys_coeff_optimControl}. Since $(K\i,\Lambda\i, Y\i, Z\i, R\i)$ is solution to \eqref{sys_coeff_optimControl}, the drift of $t \mapsto \E[\cali{S}^{\cbold,i}]$ is positive for all $i \in \inter$ and all $\cbold \in \cspace$, which implies condition (iii). Finally, for $\cbold \in \cspace$, we see that $\E[D_t^{\cbold,i}] \equiv 0$ for $t \in [0,T]$ and $i \in \inter$ if and only if:
\begin{equation*}
\c_{i,t} - \bar{\c} _{i,t} - a^{i,1}_t(X_t^{\cbold},\E[X_t^{\cbold}])+ \theta^{i \intercal}_t( \bar{\c}_{i,t} - a^{i,0}_t(\E[X_t^{\cbold}]) ) =0  \hspace{1cm} a.s. \hspace{1cm} t \in [0,T]. 
\end{equation*}
Since $\theta\i_t$ is invertible, we get $  \bar{\c}_{i,t} = a^{i,0}_t $ by taking the expectation in the above formula. Thus $\E[D_t^{\cbold,i}] \equiv 0$ for every $i \in \inter$ and $t \in [0,T]$ if and only if $\c_{i,t} = \bar{\c}_{i,t} + a^{i,1}_t= a^{i,1}_t + a^{i,0}_t$ for every $i \in \inter$ and $t \in [0,T]$. For such controls for the players, the condition (iv)  is satisfied.
We now check that $\c_i \in \cali{A}\i$ for every $i \in \inter$ (i.e. it satisfies the square integrability condition). Since $X$ is solution to a linear Mckean-Vlasov dynamics and satisfies the square integrability condition $\E[\sup_{0\leq t \leq T} |X_t|^2] < \infty$, it implies that $\c_i \in L^2_{\mathbb{F}} (\Omega \times [0,T], \R^{d_i})$ since $S\i_i, U\i_i, \hat{S}\i_i, V\i_i$ are bounded and $(O\i_i, \xi\i_i) \in L^2([0,T], \R^{d_i})\times L^2(\Omega \times [0,T], \R^{d_i})$. Therefore $\c_i \in \cali{A}\i$ for every $i \in \inter$. 
\end{proof}


\subsubsection{Step 5:  search for a fixed point}
We now find semi-explicit expressions for the optimal controls of each player. The issue here is the fact that the controls of the other players appear in the best response functions of each player through the vectors $(Y^1, Z^1),...,(Y^n,Z^n)$. To solve this fixed point problem, we first rewrite \eqref{best_response} and the backward equations followed by $\boldsymbol{(Y,Z)} = ((Y^1, Z^1),...,(Y^n,Z^n))$ in the following way (note that we omit the time dependence of the coefficients to make the notations less cluttered):
\begin{equation}
\begin{cases}
\cbstar_t - \bar{\cbold^\star}_t &= \boldsymbol{S_x}(X_t - \bar{X}_t) + \boldsymbol{S_y(Y_t - \bar{Y}_t)} + \boldsymbol{S_z(Z_t - \bar{Z}_t)} + \boldsymbol{H - \bar{H}} \\
\bar{\cbold^\star}_t &= \boldsymbol{\hat{S}_x} \bar{X}_t + \boldsymbol{\hat{S}_y \bar{Y}_t } + \boldsymbol{\hat{S}_z \bar{Z}_t } + \boldsymbol{\bar{\hat{H}}} \\
d \boldsymbol{Y}_t &= \left( \boldsymbol{ P_y(Y_t - \bar{Y}_t) + P_z(Z_t - \bar{Z}_t) + P_{\c}(\c_t - \bar{\c}_t) + F - \bar{F}     }  \right.\\
&+ \left. \boldsymbol{ \hat{P}_y\bar{Y}_t + \hat{P}_z \bar{Z}_t + \hat{P}_{\c} \bar{\c}_t + \bar{\hat{F}}     } \right)dt \\
& + \boldsymbol{Z}_t dW_t,
\end{cases}
\label{initial system}
\end{equation}
where we define 
\begin{align}
\begin{cases}
\boldsymbol{S} &=  \entrepar{ (S_i^{i})^{-1}\1_{i=j}}_{i,j\in\inter} \\
\boldsymbol{\hat{S}} &=  \entrepar{(\hat{S}_i^{i})^{-1}\1_{i=j}}_{i,j\in\inter} \\
\boldsymbol{J} &= \frac{1}{2}\entrepar{ (J\i_{ij} + J\i_{ji})\1_{i\neq j} }_{i,j\in \inter }  \\
\boldsymbol{\hat{J}} &= \frac{1}{2}\entrepar{ (\hat{J}\i_{ij} + \hat{J}\i_{ji})\1_{i\neq j} }_{i,j\in \inter }  \\
\boldsymbol{\cali{J}} &= -\entrepar{I_d +  \boldsymbol{S} \boldsymbol{J}   }^{-1}\boldsymbol{S}  \\
\boldsymbol{\hat{\cali{J}}} &= -\entrepar{I_d +  \boldsymbol{\hat{S}} \boldsymbol{\hat{J}}   }^{-1}\boldsymbol{\hat{S}}  \\
\boldsymbol{S_x} &= \boldsymbol{\cali{J}} \entrepar{U\i_i}_{i \in \inter} \\
\boldsymbol{\hat{S}_x} &= \boldsymbol{\hat{\cali{J}} } \entrepar{V\i_i}_{i \in \inter} \\
\boldsymbol{S_y} &= \boldsymbol{\cali{J}} \entrepar{\1_{i=j}b_i^\intercal}_{i,j \in \inter} \\
\boldsymbol{\hat{S}_y} &= \boldsymbol{\hat{\cali{J}}} \entrepar{ 1_{i=j} \hat{b}_i^\intercal}_{i,j \in \inter} \\
\boldsymbol{S_z} &= \boldsymbol{\cali{J}} \entrepar{\sigma_i^\intercal}_{i \in \inter} \\
\boldsymbol{\hat{S}_z} &= \boldsymbol{\hat{\cali{J}}} \entrepar{\hat{\sigma}_i^\intercal}_{i \in \inter} \\
\boldsymbol{H} &= \boldsymbol{\cali{J}} \entrepar{L\i_i + \sigma^\intercal_i K\i\gamma }_{i \in \inter} \\
\boldsymbol{\hat{H}} &= \boldsymbol{\hat{\cali{J}}} \entrepar{ L\i_i + \hat{\sigma}^\intercal_i K\i\gamma }_{i \in \inter} 
\end{cases}
&
\begin{cases}
\boldsymbol{P_y} &= (\1_{i=j} (U\i_i(S^{i}_i)^{-1}b_i^\intercal - b_x^\intercal+\rho) )_{i,j \in \inter}\\
\boldsymbol{\hat{P}_y} &= (\1_{i=j} (V\i_i (\hat{S}^{i}_i)^{-1} \hat{b}_i^\intercal - \hat{b}_x^\intercal+\rho) )_{i,j \in \inter} \\
\boldsymbol{P_z} &= (\1_{i=j} (U\i_i (S^{i}_i)^{-1} \sigma_i^\intercal - \sigma_x^\intercal) )_{i,j \in \inter} \\
\boldsymbol{\hat{P}_z} &= (\1_{i=j} (V\i_i (S^{i}_i)^{-1} \hat{\sigma}_i^\intercal - \hat{\sigma}_x^\intercal) )_{i,j \in \inter}  \\
\boldsymbol{P_{\c}} &=  -(\1_{i\neq j}(U\i_j + U\i_i (S^{i}_i)^{-1}(J\i_{ij}+J^{i\intercal}_{ji})) )_{i,j \in \inter}\\
\boldsymbol{\hat{P}_{\c}} &= -(\1_{i\neq j}(V\i_j + V\i_i (\hat{S}^{i}_i)^{-1} (\hat{J}\i_{ij}+\hat{J}^{i\intercal}_{ji})) )_{i,j \in \inter}\\
\boldsymbol{F} &= (K\i \beta + \sigma_x^\intercal K\i \gamma )_{i \in \inter}\\
\boldsymbol{\hat{F}} &= (U\i_i (S^{i}_i)^{-1}(L_i+\sigma_i^\intercal K\i \gamma) - L_x - \sigma_x^\intercal K\i \gamma - K\i \beta )_{i \in \inter}. 
\end{cases}
\label{coeff_ansatz}
\end{align}

Now, the strategy  is to propose an ansatz for $t\in [0, T] \mapsto\boldsymbol{Y}_t$ in the form: 
\begin{equation}
\boldsymbol{Y}_t = \pi_t (X_t - \bar{X}_t) + \hat{\pi}_t \bar{X}_t + \eta_t
\label{ansatz}
\end{equation}
where $(\pi, \hat{\pi}, \eta) \in L^{\infty}([0,T], \R^{nd \times d}) \times L^{\infty}([0,T], \R^{nd \times d}) \times \cali{S}_\F^2(\Omega\times [0,T], \R^{nd})$ satisfy:
\begin{equation*}
\begin{cases}
d\eta_t &= \psi_t dt + \phi_t dW_t, \;\;\; \eta_T = \boldsymbol{r} = (r^i)_{i\in \inter} \\
d\pi_t &= \dot{\pi}_t dt, \;\;\; \pi_T = 0 \\
d\hat{\pi}_t &= \dot{\hat{\pi}}_t dt, \;\;\; \hat{\pi}_T =0. 
\end{cases}
\end{equation*}
By applying It\^o's formula to the ansatz we then obtain:
\begin{equation*}
\begin{split}
d\boldsymbol{Y}_t &= \dot{\pi}_t(X_t - \bar{X}_t) dt + \pi_t d(X_t-\bar{X}_t) + \dot{\hat{\pi}}_t\bar{X}_t dt +\hat{\pi}_td\bar{X}_t + \psi_t dt + \phi_t dW  \\
&=  dt \entrecro{\dot{\pi}_t(X_t - \bar{X}_t) + \psi_t - \bar{\psi}_t + \pi_t  \entrepar{\beta - \bar{\beta} + b_x (X_t - \bar{X}_t) +  B (\cbold_t - \bar{\cbold}_t) }  } \\
&+ dt \entrecro{\dot{\hat{\pi}}_t\bar{X}_t + \bar{\psi}_t + \hat{\pi}_t \entrepar{ \bar{\beta} + \hat{b}_x \bar{X}_t +  \hat{B}  \bar{\cbold}_t }} \\
&+ dW_t \entrecro{ \phi_t + \pi_t\entrepar{\gamma + \sigma_x X_t + \tilde{\sigma}_{x} \bar{X}_t + \Sigma \cbold_t + \tilde{\Sigma} \bar{\cbold}_t}  }.
\end{split}
\end{equation*}
By comparing the two It\^o's decompositions of $\boldsymbol{Y}$, we get 
\begin{equation}
\begin{cases}
\boldsymbol{ P_y(Y_t - \bar{Y}_t) + P_z(Z_t - \bar{Z}_t) + P_{\c}(\c_t - \bar{\c}_t) + F - \bar{F}} &= \dot{\pi}_t(X_t-\bar{X}_t) + \psi_t - \bar{\psi}_t\\
& + \pi_t  \entrepar{\beta - \bar{\beta} + b_x (X_t - \bar{X}_t) +  B (\cbold_t - \bar{\cbold}_t) }  \\
\boldsymbol{ \hat{P}_y\bar{Y}_t + \hat{P}_z \bar{Z}_t + \hat{P}_{\c} \bar{\c}_t + \bar{\hat{F}}     } &= 
\entrecro{\dot{\hat{\pi}}_t\bar{X}_t + \bar{\psi}_t + \hat{\pi}_t \entrepar{ \bar{\beta} + \hat{b}_x \bar{X}_t +  \hat{B}  \bar{\cbold}_t }} \\
\boldsymbol{Z}_t &= \entrecro{ \phi_t + \pi_t \entrepar{\gamma + \sigma_x X_t + \tilde{\sigma}_{x} \bar{X}_t + \Sigma \cbold_t + \tilde{\Sigma} \bar{\cbold}_t} }. 
\end{cases}
\label{identification_y}
\end{equation}
We now substitute the $\boldsymbol{Y}$  by its ansatz in the best response equation \eqref{initial system}, and obtain the system:
\begin{equation}
\begin{cases}
(Id - S_z \pi \Sigma) (\cbstar_t - \bar{\cbold^\star}_t) &= (S_x+S_y \pi_t + S_z \pi_t \sigma_x)(X_t - \bar{X}_t)\\
&+ (H-\bar{H} + S_y(\eta_t - \bar{\eta}_t) + S_z(\phi_t - \bar{\phi}_t + \pi_t(\gamma - \bar{\gamma}))) \\      
(Id - \hat{S}_z \pi_t \hat{\Sigma})\bar{\cbold^\star}_t &= (\hat{S}_x+ \hat{S}_y \hat{\pi}_t +\hat{S}_z \pi_t \hat{\sigma}_x)\bar{X}_t + ( \bar{\hat{H}} + \hat{S}_y \bar{\eta}_t + \hat{S}_z (\bar{\phi}_t + \pi_t.  \bar{\gamma})). 
\end{cases}
\label{control_new_style}
\end{equation}
To make the next computations slightly less painful we rewrite \eqref{control_new_style} as 
\bec{
\cbstar_t - \bar{\cbstar}_t &= A_x(X_t - \bar{X}_t) + R_t - \bar{R}_t \\
\bar{\cbstar}_t &= \hat{A}_x\bar{X}_t + \bar{\hat{R}}_t \\
&\textnormal{where} \\
A_x &:= (Id - S_z \pi_t \Sigma)^{-1} (S_x + S_y\pi_t + S_z\pi_t \sigma_x) \\
\hat{A}_x &:= (Id - \hat{S}_z \pi_t \hat{\Sigma})^{-1} (\hat{S}_x + \hat{S}_y\pi_t + \hat{S}_z\pi_t \hat{\sigma}_x) \\
R_t  &:=(Id - S_z \pi_t \Sigma)^{-1} (H+ S_y\eta_t  + S_z(\phi_t  + \pi_t\gamma ))  \\
\hat{R}_t &:= (Id - \hat{S}_z \pi_t \hat{\Sigma})^{-1} ( {\hat{H}} + \hat{S}_y {\eta} + \hat{S}_z (\phi_t + \pi_t {\gamma})). 
\label{feedback_control_complete}}
By injecting \eqref{control_new_style} into \eqref{identification_y} we have:
\begin{equation}
\begin{cases}
0 &= \entrecro{ \dot{\pi}_t +\pi_t b_x-P_y\pi_t-P_z\pi_t(\sigma_x + \Sigma A_x) - (P_{\c}-\pi_t B)A_x }(X_t - \bar{X}_t ) \\
&+ \psi_t - \bar{\psi}_t + \pi_t(\beta - \bar{\beta}) - (P_{\c} - \pi_t B)(R-\bar{R}) - (F-\bar{F})\\
&- P_z(\phi_t - \bar{\phi}_t+ \pi_t(\gamma - \bar{\gamma} + \Sigma(R-\bar{R}) ) ) - P_y(\eta_t - \bar{\eta}_t) \\
0 &= \entrecro{ \dot{\hat{\pi}}_t + \hat{\pi}_t \hat{b}_x - \hat{P}_y\hat{\pi}-\hat{P}_z\pi_t(\hat{\sigma}_x + \hat{\Sigma} \hat{A}_x) - (\hat{P}_{\c}-\hat{\pi}_t \hat{B})\hat{A}_x } \bar{X}_t \\
&+  \bar{\psi}_t + \hat{\pi}_t \bar{\beta} - (\hat{P}_{\c} - \hat{\pi}_t \hat{B})\bar{\hat{R}} - \bar{\hat{F}} - \hat{P}_z( \bar{\phi}_t+ \pi_t( \bar{\gamma} + \hat{\Sigma}\bar{\hat{R}} ) ) - \hat{P}_y \bar{\eta}. 
\end{cases}
\end{equation}
Thus we constrain the coefficients $(\pi, \hat{\pi}, \psi, \phi)$ of the ansatz of $\boldsymbol{Y}$  to satisfy:
\begin{equation}
\begin{cases}
\dot{\pi}_t &= -\pi b_x + P_y\pi_t +P_z\pi_t\sigma_x + (P_{\c} + P_z\Sigma)  (Id - S_z \pi_t \Sigma)^{-1} (S_x + S_y\pi_t + S_z\pi_t \sigma_x)  \\
& -\pi_t B  (Id - S_z \pi_t \Sigma)^{-1} (S_x + S_y\pi_t + S_z\pi_t \sigma_x) \\
\pi_T &= 0 \\
\dot{\hat{\pi}}_t &= -\hat{\pi}_t \hat{b}_x + \hat{P}_y\hat{\pi}_t + \hat{P}_z\pi_t\hat{\sigma}_x + (\hat{P}_{\c} + \hat{P}_z\hat{\Sigma} )(Id - \hat{S}_z \pi_t \hat{\Sigma})^{-1} (\hat{S}_x + \hat{S}_y\hat{\pi}_t + \hat{S}_z\pi_t \hat{\sigma}_x) \\
&-\hat{\pi}_t \hat{B}(Id - \hat{S}_z \pi_t \hat{\Sigma})^{-1} (\hat{S}_x + \hat{S}_y\hat{\pi}_t + \hat{S}_z\pi_t \hat{\sigma}_x)\\
\hat{\pi}_T &= 0 \\
d\eta_t &= \psi_t dt + \phi_t dW \\
\eta_T &= \boldsymbol{r} \\
&\textnormal{where:}\\
\psi_t - \bar{\psi}_t &=  -\pi_t(\beta - \bar{\beta}) + (P_{\c} - \pi_t B)(R-\bar{R}) + (F-\bar{F})\\
&+ P_z(\phi_t - \bar{\phi}_t + \pi_t(\gamma - \bar{\gamma} + \Sigma(R-\bar{R}) ) ) + P_y(\eta_t - \bar{\eta}_t)  \\
\bar{\psi}_t &= - \hat{\pi}_t \bar{\beta} + (\hat{P}_{\c} - \hat{\pi}_t \hat{B})\bar{\hat{R}} + \bar{\hat{F}} + \hat{P}_z( \bar{\phi}_t + \pi_t (\bar{\gamma} + \hat{\Sigma}\bar{\hat{R}} ) ) + \hat{P}_y \bar{\eta}_t \\
R_t  &:=(Id - S_z \pi \Sigma)^{-1} (H+ S_y\eta  + S_z(\phi  + \pi\gamma ))  \\
\hat{R}_t &:= (Id - \hat{S}_z \pi \hat{\Sigma})^{-1} ( {\hat{H}} + \hat{S}_y {\eta} + \hat{S}_z ({\phi} + \pi {\gamma})). 
\end{cases}
\label{sys_fixed_point}
\end{equation}
We now have a feedback form for $\boldsymbol{(Y,Z)} = ((Y^1,Z^1),...,(Y^n,Z^n))$. We can inject it in the best response functions $\cbstar$ in order to obtain the optimal controls in feedback form. We then inject these latter in the state equation in order to obtain an explicit expression of $t \mapsto X_t^\star$.

\subsubsection{Step 6: check the validity}

Let us now check  the existence and uniqueness of  $t \mapsto (K\i_t, \Lambda\i_t, (Y\i_t, Z_t\i), R\i_t, \pi_t, \hat{\pi}_t, (\eta_t, \phi_t))$ where  $K\i \in L^{\infty}([0,T], \mathbb{S}^d_+)$, 
$\Lambda\i \in L^{\infty}([0,T], \mathcal{S}^d_+) $, $(Y\i_t, Z_t\i) \in  \cali{S}^2_{\mathbb{F}}(\Omega\times [0,T],\R^{d}) \times L^2_{\mathbb{F}}(\Omega\times [0,T], \R^{d})$, 
$R\i \in L^{\infty}([0,T], \R)$, $\pi, \hat{\pi} \in L^{\infty}([0,T], \R^{nd \times d})$ and $(\eta, \phi) \in \cali{S}^2_{\mathbb{F}}(\Omega\times [0,T],\R^{nd})  \times L^2_{\mathbb{F}}(\Omega\times [0,T], \R^{nd})$, 
under the assumptions  {\bf (H1)}-{\bf (H2)}. We recall that $t \mapsto (K\i_t, \Lambda\i_t, (Y\i_t, Z_t\i), R\i_t$  and $t\mapsto (\pi_t, \hat{\pi}_t, (\eta_t, \phi_t))$ are solutions respectively to \eqref{sys_coeff_optimControl} and  \eqref{sys_fixed_point}. Fix  $i \in \inter$: 
\begin{enumerate}[label=(\roman*)]
\item We first consider the coefficients $K\i$ which follow Ricatti equations:
\bec{\dot{K}_t\i  +Q\i _t + \sx^\intercal K\i_t \sx + K\i_t \bx + \bx^\intercal K\i_t - \rho K\i_t \\
- ( I_{k,t}\i + \sigma_{k,t}^\intercal K\i_t \sigma_{x,t} + b_{k,t}^\intercal K\i_t)(N\i_{k,t} + \sigma_{k,t}^\intercal K\i_t \sigma_{k,t})^{-1}( I_{k,t}\i + \sigma_{k,t}^\intercal K\i_t \sigma_{x,t} + b_{k,t}^\intercal K\i_t)  &=0 \\
K_T\i &= P\i. }
By standard result in control theory (see \cite{YZ}, Ch. 6, Thm 7.2]) under \textbf{(H1)} and \textbf{(H2)} there exists a unique solution $K\i \in L^{\infty}([0,T], \mathbb{S}^d_+).$
\item Given $K\i$ let us now consider the $\Lambda\i$'s. They also follow Ricatti equations:
\bec{
\dot{\Lambda}\i + \hat{Q}^{iK}_t + \Lambda\i_t \hat{b}_{x,t} + \hat{b}_{x,t}^\intercal \Lambda\i_t - \rho \Lambda\i_t - (\hat{I}^{iK}_t + \hat{b}_{i,t}^\intercal \Lambda \i_t)(\hat{N}^{iK}_t)^{-1} (\hat{I}^{iK}_t + \hat{b}_{i,t}^\intercal \Lambda \i_t)&=0   \\
\Lambda\i_T &= \hat{P}\i}
where we define:
\bec{
\hat{Q}^{iK}_t &:= \hat{Q}\i _t + \hat{\sigma}_{x,t}^\intercal K\i_t \hat{\sigma}_{x,t} \\
\hat{I}^{iK}_t &:= \hat{I}_{i,t}\i + \hat{\sigma}_{i,t}^\intercal  K\i_t \hat{\sigma}_{x,t}\\
\hat{N}^{iK}_t &:= \hat{N}\i_{i,t} + \hat{\sigma}_{i,t}^\intercal K\i_t  \hat{\sigma}_{i,t}. 
}
We need the same arguments as for $K\i$. The only missing argument to conclude the existence and uniqueness of $\Lambda\i$ is: $\hat{Q}^{iK} -(\hat{I}^{iK})^{\intercal} (\hat{N}^{iK})^{-1} \hat{I}^{iK} \geq 0$. As in \cite{BP} we can prove with some algebraic calculations that it is implied by the hypothesis $\hat{Q}^{i} -(\hat{I}^{i})^{\intercal} (\hat{N}^{i})^{-1} \hat{I}^{i} \geq 0$ that we made in {\bf (H2)}. 
\item Given $(K\i, \Lambda\i)$ we now consider the equation for $(Y\i, Z\i)$ which is a linear mean-field BSDE of the form:
\bec{
dY\i_t &= \cali{V}\i_t + \cali{G}^{i}_t(Y\i_t - \E[Y\i_t]) + \cali{\hat{G}}^{i}_t \E[Y\i_t] + \cali{J}^{i}_t(Z\i_t - \E[Z\i_t]) + \cali{\hat{J}}^{i}_t \E[Z\i_t]  + Z\i_t dW_t  \label{sys_coeff_optimControl2} \\ 
Y\i_T &= r\i,
}
where the deterministic coefficients $\cali{G}\i_t, \hat{\cali{G}}\i, \cali{J}\i, \hat{\cali{J}}\i \in L^\infty([0, T], \R^{d\times d})$ and the stochastic process $\cali{V}\i \in L^2([0,T], \R^{d})$ are defined as:
\bec{
\cali{V}\i_t &:= -L\i_{x,t} - \Lambda\i_t \bar{\beta}_t - \sx^\intercal K\i_t \gamma_t - \sxt^\intercal K\i_t \bar{\gamma}_t - K\i_t(\beta_t - \bar{\beta}_t) - \sum_{k \neq i} \entrepar{ U\i_{k,t} (\c_{k,t} - \bar{\c}_{k,t}) + V\i_{k,t} \bar{\c}_{k,t} }  \\
&+ U_{i,t}^{i \intercal} S_{t}^{i-1} (L\i_{k,t} - \E[L\i_{k,t}] + \sigma_{i,t}^\intercal K\i_t ( \gamma_t  - \EE{\gamma_t})  +  \frac{1}{2} \sum_{k \neq i} (J\i_{i,k,t} + J_{k,i,t}^{i\intercal})(\c_{k,t} - \EE{\c_{k,t}} ) ) \\
&+ V_{i,t}^{i \intercal} \hat{S}_{t}^{i-1} (\E[{L}\i_{k,t}] + \hat{\sigma}_{i,t}^\intercal K\i_t - \EE{\gamma_t}  +  \frac{1}{2} \sum_{k \neq i} (\hat{J}\i_{i,k,t} + \hat{J}_{k,i,t}^{i\intercal}) \EE{\c_{k,t}} )  \\
\cali{G}\i_t &:= \rho I_d- \bx^\intercal + U_{i,t}^{i \intercal} S_{i,t}^{i-1} b_{i,t}^\intercal \\
\cali{\hat{G}}\i_t &= \rho I_d- \hat{b}_{x,t}^\intercal + V_{i,t}^{i \intercal} \hat{S}_{i,t}^{i-1} \hat{b}_{i,t}^\intercal \\
\cali{J}\i_t &:= - \sx^\intercal +  U_{i,t}^{i \intercal} S_{i,t}^{i-1} \sigma_{k,t}^\intercal \\
\cali{\hat{J}}\i_t &:= - \hat{\sigma}_{x,t}^\intercal +  V_{i,t}^{i \intercal} \hat{S}_{i,t}^{i-1} \sigma_{k,t}^\intercal.
\label{coeff_Y_bsde}
}
By standard results (see Thm. 2.1 in \cite{Li2017}) we obtain that there exists a unique solution $(Y\i, Z\i) \in \cali{S}_\F^2(\Omega\times [0,T], \R^{d \times d}) \times L_\F^2(\Omega\times [0,T], \R^{d \times d})$ to \eqref{sys_coeff_optimControl2}.
\item Given $(K\i, \Lambda\i, Y\i, Z\i)$ we consider the equation of $R\i$ which is a linear ODE whose solution is given by:
\begin{equation*}
R\i_t = \int_t^T e^{-\rho(s-t)}h\i_s \dd s,
\end{equation*}
where $h\i$ is a deterministic function defined by:
\begin{equation}
\begin{split}
h\i_t &= -\E[\gamma_t^\intercal K\i_t \gamma_t + 2 \beta_t^\intercal Y\i_t + 2 \gamma_t^\intercal Z\i_t 
+\\
& \sum_{k \neq i} (\c[_{k,t}] - \bar{\c}_{k,t})^\intercal S\i_{k,t} (\c[_{k,t}] - \bar{\c}_{k,t}) + \bar{\c}_{k,t}^\intercal \hat{S}\i_{k,t} \bar{\c}_{k,t} +2 [O^{i}_{k,t} + \xi^{i}_{k,t} - \bar{\xi}^{i}_{k,t}]^\intercal \c[_{k,t}]  ] \\
&+ \EE{(\xi\i_{i,t} - \bar{\xi}\i_{i,t})^\intercal S_t^{i-1}(\xi\i_{i,t} - \bar{\xi}\i_t) - O_{i,t}^{i\intercal} \hat{S}_t^{i-1} O_{i,t}\i},
\end{split}
\label{coeff_R}
\end{equation}
and  the expressions of the coefficients are recalled in \eqref{coeffA_a}.
\item The final step is to verify that the procedure to find a fixed point is valid. More precisely we need to ensure that $t \mapsto (\pi_t, \hat{\pi}_t, \eta_t)$ is well defined. It is difficult to ensure the well posedness of $t \mapsto (\pi_t, \hat{\pi}_t)$ for two reasons: first because $\pi$ and $\hat{\pi}$ follow Ricatti equations but are not squared matrices; and second because $\pi$ appears in the equation followed by $\hat{\pi}$. We are not aware of any work addressing this kind of equations in a general setting.
\newline
If we suppose $t \mapsto (\pi_t, \hat{\pi}_t)$ well defined, then $t \mapsto (\eta_t,\phi_t)$ follows a linear mean-field BSDE of the type:
\bec{
d\eta_t &= \entrepar{\cali{V}_t + \cali{G}_t(\eta_t - \E[\eta_t]) + \cali{\hat{G}}_t \EE{\eta_t} + \cali{J}_t(\phi_t - \E[\phi_t]) + \cali{\hat{J}}_t \EE{\phi_t} } dt + \phi_t dW_t\\
\eta_T &= \boldsymbol{r},
\label{bsde_eta}}
where the deterministic coefficients $\cali{G}\i_t, \hat{\cali{G}}\i, \cali{J}\i, \hat{\cali{J}}\i \in L^\infty([0, T], \R^{nd\times nd})$ and the stochastic process $\cali{V}\i \in L^2([0,T], \R^{nd})$ are defined as:
\bec{
\cali{V}_t &:=  -\pi(\beta - \bar{\beta}) + F - \bar{F} + P_z\pi(\gamma - \bar{\gamma}) +  (P_{\c} - \pi B + P_z \pi \Sigma)(I_d - S_z\pi \Sigma)^{-1}(H-\bar{H} + S_z \pi(\gamma - \bar{\gamma})) \\
&- \hat{\pi} \bar{\beta} + \bar{\hat{F}} + \hat{P}_z\pi \bar{\gamma} +  (\hat{P}_{\c} - \hat{\pi} \hat{B} + \hat{P}_z \pi \hat{\Sigma})(I_d - \hat{S}_z\pi \hat{\Sigma})^{-1}(\bar{\hat{H}} +  \hat{S}_z \hat{\pi} \bar{\gamma})\\
\cali{G}_t &:= \entrecro{P_y + (P_{\c} - \pi B + P_z \pi \Sigma)((I_d - S_z\pi \Sigma)^{-1} S_y)} \\
\cali{\hat{G}}_t &= \hat{P}_y \\
\cali{J}_t &:= \entrecro{P_z + (P_{\c} - \pi B + P_z \pi \Sigma)((I_d - S_z\pi \Sigma)^{-1} S_z)}\\
\cali{\hat{J}}_t &:= - \hat{\sigma}_{x,t}^\intercal +  V_{i,t}^{i \intercal} \hat{S}_{i,t}^{i-1} \sigma_{k,t}^\intercal.
}
Again, by standard results (see Thm. 2.1 in \cite{Li2017}) we obtain that there exists a unique solution $(\eta, \phi) \in \cali{S}^2(\Omega\times [0,T], \R^{d \times d}) 
\times L_{\mathbb F}^2(\Omega\times [0,T], \R^{d \times d})$ to \eqref{bsde_eta}.
\end{enumerate}

To sum up the arguments previously presented,  our main result provides the following characterization of the Nash equilibrium:

\begin{theorem}
Suppose assumptions \textbf{(H1)} and \textbf{(H2)}. Suppose also that the system associated with the fixed point search \eqref{sys_fixed_point} is well defined. Then $\cbstar = (\c_1, ...,\c_n)$ defined by 
\begin{equation}
\begin{cases}
\cbstar_t - \bar{\cbold^\star}_t &= \boldsymbol{S_{x,t}}(X_t^\star - \bar{X}_t^\star) + \boldsymbol{S_{y,t}(Y_t - \bar{Y}_t)} + \boldsymbol{S_{z,t}(Z_t - \bar{Z}_t)} + \boldsymbol{H_t - \bar{H}_t} \\
\bar{\cbold^\star}_t &= \boldsymbol{\hat{S}_{x,t}} \bar{X}_t^\star + \boldsymbol{\hat{S}_{y,t} \bar{Y}_t } + \boldsymbol{\hat{S}_{z,t} \bar{Z}_t } + \boldsymbol{\bar{\hat{H}}_{t}} 
\end{cases}
\end{equation}
where  $\boldsymbol{S_x, S_y, S_z, H, \hat{S}_x, \hat{S}_y, \hat{S}_z, \hat{H}}$ are defined in \eqref{coeff_ansatz}, $\boldsymbol{(Y,Z)}$ are in $\mathcal{S}_{\mathbb{F}}^{2}(\Omega \times [0,T], \R^{nd\times d}) \times L^{2}_{\mathbb{F}}( \Omega \times [0,T], \R^{nd\times d})$ and satisfy \eqref{ansatz} and \eqref{sys_fixed_point}, is a Nash equilibrium.
\label{theorem_nash_eq}
\end{theorem}


\section{Some extensions} \label{secexten}

\subsection{The case of infinite horizon}

Let us now tackle the infinite horizon case. The method is similar to the finite-horizon case but some adaptations are needed when dealing with the well posedness of $(K\i, \Lambda\i, Y\i, R\i)$ and the admissibility of the controls. 

We redefine the set of admissible controls for for each player $i \in \inter $ as:

\begin{equation}
\mathcal{A}\i = \left\{ \c : \Omega \times [0,T] \to \R^{d_i}  \;  \text{ s.t. \;} \c \text{\;is $\mathbb{F}$-adapted and} \int_0^\infty e^{-\rho u} \E[|\c_u|^2] du < \infty   \right\}
\end{equation}
while the controlled state defined on $\R_+$ now follows a dynamics of the form 
\bec{
d\X &= b \args dt + \sigma \args dW_t \\
X_0^{\c} &= X_0
\label{dynamic_inf}
}
where  for each $t \in [0,T]$, $x, \bar{x} \in \R^d$, $a_i, \bar{a}_i \in \R^{d_i}$:
\bec{
b(t,x,\bar{x}, \boldsymbol{a}, \bar{\boldsymbol{a}}) &= \beta_t + b_x x + \tilde{b}_{x}\bar{x} + \sum_{i=1}^n b_{i} a_{i} + \tilde{b}_{i} \bar{a}_{i}  \\
&= \beta_t + b_x x + \tilde{b}_{x}\bar{x} + B \boldsymbol{a}_t + \tilde{B} \bar{\boldsymbol{a}}_t \\
\sigma (t,x,\bar{x}, \boldsymbol{a}, \bar{\boldsymbol{a}})  &= \gamma_t + \sigma_x x + \tilde{\sigma}_{x}\bar{x} + \sum_{i=1}^n \sigma_{i} a_{i} + \tilde{\sigma}_{i} \bar{a}_{i} \\
&= \gamma_t + \sigma_x x + \tilde{\sigma}_{x}\bar{x} + \Sigma \boldsymbol{a} + \tilde{\Sigma} \bar{\boldsymbol{a}}.
\label{coefX_inf}
}

Notice that now only the coefficients $\beta$ and $\gamma$ are allowed to be stochastic processes. The other linear coefficients are constant matrices.

The goal of each player $i \in \inter $ during the game is still to minimize her  cost functional with respect to control $\c_i$ over $\mathcal{A}\i$, and given control $\alpha^{-i}$ of the other players: 
\begin{equation}
J\i(\alpha_i,\alpha^{-i})  = \E \left[ \int_0^\infty e^{-\rho t} f\i \args  dt   \right]
\label{J_inf}
\end{equation}
where for each $t \in [0,T]$, $x, \bar{x} \in \R^d$, $a_i, \bar{a}_i \in \R^{d_i}$ we have set the running cost for each player as:
\begin{equation}
\begin{cases}
f\i (t,x,\bar{x}, \boldsymbol{a}, \boldsymbol{\bar{a}}) &= (x-\xbar )^\intercal Q\i (x-\xbar ) + \xbar ^\intercal[Q\i  + \tilde{Q}\i] \xbar  \\
& + \sum_{k=1}^n a_k^\intercal I_{k}\i (x-\bar{x}) + \bar{a}_k^\intercal (I_{k}\i + \tilde{I}_{k}\i ) \bar{x}  \\
& + \sum_{k=1}^n (a_k - \bar{a}_k)^\intercal N_{k}\i (a_k - \bar{a}_k) + \bar{a}_k (N_{k}\i + \tilde{N}_{k,}\i ) \bar{a}_k \\
& +  \sum_{0 \le k \ne l \le n} (a_k - \bar{a}_k)^\intercal G\i_{k,l} (a_l - \bar{a}_l) +  a_k^\intercal ( G\i_{k,l} + \tilde{G}_{k,l}\i  )a_l\\
& + 2[L_{x,t}\i[\intercal] x + \sum_{k=1}^n L_{k,t}^{i\intercal} a_k]. 
\end{cases}
\label{coefJ_inf}
\end{equation}

Note that the only coefficients that we allow to be time dependent are $L\i_x$ and $L\i_k$ for $k \in \inter$ which may be stochastic processes.

\subsubsection{Assumptions}

We detail below the new assumptions:

\vspace{2mm}

\noindent \textbf{(H1')} The coefficients in \eqref{coefX_inf} satisfy:
\begin{enumerate}[label= \alph*),leftmargin=2cm ,parsep=0cm,itemsep=0cm,topsep=0cm]
    \item $\beta, \gamma \in L^{2}_{\mathbb{F}}( \Omega \times [0,T], \R^d)  $
    \item $b_x, \tilde{b}_x, \sigma_x, \tilde{\sigma}_x  \in  \R^{d\times d} $; $b_i, \tilde{b}_i, \sigma_i, \tilde{\sigma}_i \in \R^{d\times d_i}$
\label{H1'}
\end{enumerate}
\noindent
\textbf{(H2')} The coefficients of the cost functional \eqref{coefJ} satisfy:
\begin{enumerate}[label= \alph*),leftmargin=2cm ,parsep=0cm,itemsep=0cm,topsep=0cm]
  \item $Q\i, \tilde{Q}\i \in \mathbb{S}^d_+$; $N\i_k, \tilde{N}_k\i \in \mathbb{S}^{d_k}_+$; $I\i_k, \tilde{I}_k\i \in \R^{d_k \times d}$
  \item $L_x\i \in L^{2}_{\mathbb{F}}( \Omega \times \R^\star_+, \R^d)$, $L_k\i \in L^{2}_{\mathbb{F}}( \Omega \times \R^\star_+, \R^{d_k})$
  \item 
  $N\i_{k} > 0\;\;\;\;\;  Q\i - I_{i}^{i\intercal}(N_{i}^{i})^{-1} I_{i}^{i} \geq 0 $
   \item 
  $\hat{N}\i_{k,t} > 0 \;\;\;\;\;  \hat{Q}\i - \hat{I}_{i}^{i\intercal}(\hat{N}_{i}^{i})^{-1} \hat{I}_{i}^{i} \geq 0$
\label{H2'}
\end{enumerate}
\noindent
\textbf{(H3')} $\rho > 2\entrepar{|b_x|+|\tilde{b}_x|+8(|\sigma_x|^2 + |\tilde{\sigma}_x|^2)}$

As shown below, the new hypothesis \textbf{(H3')} ensure the well posedness of our problem. Notice first that by \textbf{(H1')} and classical results, there exists a unique strong solution $X^{\c}$ to the SDE \eqref{dynamic_inf}. Furthermore by \textbf{(H1')} and \textbf{(H3')} we obtain by similar arguments as in \cite{BP}  the following estimate: 
\begin{equation}
    \int_0^\infty e^{-\rho u} \E[|X_u^{\cbold}|^2]  du \leq C_{\c}(1+\E[|X_0|^2])<\infty, 
    \label{estimate_X_inf}
\end{equation}
in which $C_{\c}$ is a constant depending on $\cbold = (\c_1,...,\c_n)$ only through $\int_0^\infty e^{-\rho} \E[|\cbold_u|^2]\dd u$. Finally by \textbf{(H2')} and \eqref{estimate_X_inf} the minimizing problem \eqref{J_inf} is well defined for each player.

\subsubsection{A weak submartingale optimality principle on infinite horizon}
We now give an easy adaptation of the weak submartingale optimality principle in the case of infinite horizon. 

\begin{lemma}[Weak submartingale optimality principle]
\label{optimPrinciple_inf}
Suppose there exists a couple \\$(\cbold^\star , (\cali{W}^{.,i})_{i \in \inter })$,
where $\cbold^\star \in \cspace$ and  $\cali{W}^{.,i}=\{ \cali{W}^{\cbold,i}_t, t \in \R^\star_+, \cbold \in \cspace  \}$ is a family of adapted processes indexed by $\cspace$ for each $i\in \inter$, such that:

\begin{enumerate}[label=\roman*]
\item[(i)] For every $\cbold \in \cspace$, $\E[\cali{W}^{\cbold,i}_0]$ is independent of the control $\c_i \in \cspace[i]$; \label{item1_inf}
\item[(ii)] For every $\cbold \in \cspace$, $\lim_{t \to \infty} e^{-\rho t}\EE{\cali{W}_t^{\cbold,i}} = 0$; \label{item2_inf}
\item[(iii)] For every $\cbold \in \cspace$, the map $t \in \R^\star_+ \mapsto \EE{\cali{S}_t^{\cbold,i} } $, with \\ $\cali{S}_t^{\cbold,i} = e^{-\rho t } \cali{W}^{\cbold,i}_t + \int_0^t e^{-\rho u } f\i(u, X_u^{\cbold}, \P_{X_u^{\cbold}}, \cbold_u, \P_{\cbold_u}) du $ is well defined and non-decreasing; \label{item3_inf}
\item[(iv)]  The map $t \mapsto \E[\cali{S}_t^{\cbstar,i}]$ is constant for every $t \in \R^\star_+$; \label{item4_inf}
\end{enumerate}
Then $\cbold^\star$ is a Nash equilibrium and $J\i(\cbstar) = \E[\cali{W}_0^{\cbstar,i}]$. Moreover, any other Nash-equilibrium $\tilde{\cbold}$ such that $\E[\cali{W}_0^{\tilde{\cbold},i}] = \E[\cali{W}_0^{{\cbold^\star},i}]$ and $J\i(\tilde{\cbold}) = J\i({\cbold^\star})$ for any $i \in \inter$ satisfies the condition (iv). 
\label{lemma_verification_inf}
\end{lemma}
\begin{proof}
The proof is exactly the same as in Lemma \ref{optimPrinciple}. 
\end{proof}

Let us now describe the steps to follow in order to apply Lemma \ref{lemma_verification_inf}. Since they are similar to the ones in the finite-horizon case, we only report the main changes. 
\newline
\textbf{Steps 1-3}
\newline
For each player $i \in \inter$ we still search for a random field $\{w_t\i(x,\bar{x}),t\in [0,T], x,\bar{x} \in \R^d \}$ of the form $w_t\i(x,\bar{x}) =  \scal{K_t\i}{(x-\bar{x})} + \scal{\Lambda_t\i }{\bar{x}} +2 Y_t^{i \intercal}x + R\i_t$ for which the optimality principle  in Lemma \ref{lemma_verification_inf} now leads to the system:
\begin{equation}
\begin{cases}
dK\i_t = -\Phi^{i0}_t dt \\
d\Lambda\i_t = -\Psi^{i0}_t dt  \\
dY_t\i = -\Delta^{i0}_t dt + Z_t\i dW_t  \\
dR_t\i = -\Gamma_t^{i0} dt, \;\;\;  t \geq 0.  
\end{cases}
\label{sys_coeff_optimControl_inf}
\end{equation}
Notice that there are no terminal conditions anymore since we are in the infinite horizon case. The coefficients $\Phi^{i0}, \Psi^{i0}, \Delta^{i0}_t, \Gamma_t^{i0}$ are defined in \eqref{coeff_0}. The fourth step is exactly the same as in the finite horizon case. 
\newline
\textbf{Step 5}
\newline
We now search for a fixed point of the best response functions. Let us define $\boldsymbol{Y}= (Y^1,...Y^n)$ and propose an ansatz in a feedback form $Y$: 
$Y_t = \pi(X_t - \EE{X_t}) + \hat{\pi}\EE{X_t} + \eta_t$ where $\pi,\hat{\pi} \in L^{\infty}(\R_+, \R^{nd \times d})$ and $\eta \in L^{2}_{\mathbb{F}}( \Omega \times \R_+, \R^{nd})$ satisfy 
\begin{equation}
\begin{cases}
0 &= -\pi b_x + P_y\pi +P_z\pi \sigma_x + (P_{\c} + P_z\Sigma)  (Id - S_z \pi \Sigma)^{-1} (S_x + S_y\pi + S_z\pi_t \sigma_x)  \\
&-\pi B  (Id - S_z \pi \Sigma)^{-1} (S_x + S_y\pi + S_z\pi \sigma_x) \\
0 &= -\hat{\pi} \hat{b}_x + \hat{P}_y\hat{\pi} + \hat{P}_z\pi \hat{\sigma}_x + (\hat{P}_{\c} + \hat{P}_z\hat{\Sigma} )(Id - \hat{S}_z \pi \hat{\Sigma})^{-1} (\hat{S}_x + \hat{S}_y\hat{\pi} + \hat{S}_z\pi \hat{\sigma}_x) \\
&-\hat{\pi} \hat{B}(Id - \hat{S}_z \pi \hat{\Sigma})^{-1} (\hat{S}_x + \hat{S}_y\hat{\pi} + \hat{S}_z\pi \hat{\sigma}_x)\\
d\eta_t &= \psi_t dt + \phi_t dW_t \\
&\textnormal{with:}\\
\psi_t - \bar{\psi}_t &=  -\pi_t(\beta - \bar{\beta}) + (P_{\c} - \pi_t B)(R-\bar{R}) + (F-\bar{F}) + \\
&P_z(\phi_t - \bar{\phi}_t+ \pi_t(\gamma - \bar{\gamma} + \Sigma(R_t-\bar{R}_t) ) ) + P_y(\eta_t - \bar{\eta}_t)  \\
\bar{\psi}_t &= - \hat{\pi}_t \bar{\beta} + (\hat{P}_{\c} - \hat{\pi}_t \hat{B})\bar{\hat{R}}_t + \bar{\hat{F}} + \hat{P}_z( \bar{\phi}+ \pi_t( \bar{\gamma} + \hat{\Sigma}\bar{\hat{R}}_t ) ) + \hat{P}_y \bar{\eta}_t \\
R_t  &:=(Id - S_z \pi_t \Sigma)^{-1} (H+ S_y\eta_t  + S_z(\phi_t  + \pi_t\gamma ))  \\
\hat{R}_t &:= (Id - \hat{S}_z \pi_t \hat{\Sigma})^{-1} ( {\hat{H}} + \hat{S}_y {\eta}_t + \hat{S}_z ({\phi}_t + \pi_t {\gamma}))  \\
\end{cases}
\label{system_ansatz_inf}
\end{equation}
where the coefficients are defined in \eqref{coeff_ansatz}.
\newline
\textbf{Step 6}
\newline
We finally tackle the well-posedness of \eqref{sys_coeff_optimControl_inf} and \eqref{system_ansatz_inf}.
\begin{enumerate}[label=(\roman*)]
\item We first consider the ODE for $K\i$. Since the map $(t,k) \mapsto \phi_t^{i0}(k) $ does not depend on time (all the coefficient being constant) we search for a constant non-negative matrix 
$K\i \in {\mathbb S}^d$ satisfying $\Phi^{i0}(K\i)=0$, more precisely solution to:
\begin{equation}
\begin{split}
&Q\i + \sigma_x^\intercal K\i \sigma_x + K\i b_x + b_x^\intercal K\i - \rho K\i  \\
& - ( I_{k}\i + \sigma_{k}^\intercal K\i \sigma_{x} + b_{k}^\intercal K\i)(N\i_{k} + \sigma_{k}^\intercal K\i \sigma_{k})^{-1}( I_{k}\i + \sigma_{k}^\intercal K\i \sigma_{x} + 
b_{k}^\intercal K\i)  = 0.  
\label{eq_K_inf}
\end{split}
\end{equation}  
As in \cite{BP} we can show using a limit argument that there exists $K\i \in {\mathbb S}_+^d$  solution to \eqref{eq_K_inf}. The argument for $\Lambda\i$ is the same as for $K\i$.
\item Given $(K\i, \Lambda\i)$ the equation for $(Y\i,Z\i)$ is a linear mean-field BSDE on infinite horizon:
\bes{
dY\i_t &= \cali{V}\i_t + \cali{G}^{i}(Y\i_t - \E[Y\i_t]) + \cali{\hat{G}}^{i} \E[Y\i_t] + \cali{J}^{i}(Z\i_t - \E[Z\i_t]) + \cali{\hat{J}}^{i} \E[Z\i_t]  + Z\i_t dW_t,  
}
where the coefficient are defined in \eqref{coeff_Y_bsde}. Notice that now $\cali{G}^{i},\cali{\hat{G}}^{i}, \cali{J}^{i}, \cali{\hat{J}}^{i} $ are all constant matrices. To the best of our knowledge, there are no general results ensuring the existence for such equation. We then add the following assumption:

\vspace{2mm}

\textbf{(H4')} There exists a solution $(Y\i, Z\i) \in \times \mathcal{S}_{\mathbb{F}}^{2}(\Omega \times \R_+, \R^d) \times L_{\mathbb{F}}^{2}(\Omega \times \R_+, \R^d) $

\item Given $(K\i, \Lambda\i, Y\i, Z\i)$ the equation for $R\i$ is a linear ODE whose solution is:
\begin{equation*}
R\i_t = \int_t^\infty e^{-\rho(s-t)}h\i_s \; ds,
\end{equation*}
where $h\i$ is a deterministic function defined in \eqref{coeff_R}.

\item We now study the well posedness of the fixed point procedure. More precisely we need to ensure that the process $t \to (\pi, \hat{\pi}, \eta_t)$ defined as a solution of the system \eqref{system_ansatz_inf}, recalled below, is well defined. Note that in the infinite horizon framework we search for constant $\pi$ and $\hat{\pi}$.
\begin{equation}
\begin{cases}
0 &= -\pi b_x + P_y\pi +P_z\pi\sigma_x + (P_{\c} + P_z\Sigma)  (Id - S_z \pi \Sigma)^{-1} (S_x + S_y\pi + S_z\pi \sigma_x)  \\
&-\pi B  (Id - S_z \pi \Sigma)^{-1} (S_x + S_y\pi + S_z\pi \sigma_x) \\
0 &= -\hat{\pi} \hat{b}_x + \hat{P}_y\hat{\pi} + \hat{P}_z\pi \hat{\sigma}_x + (\hat{P}_{\c} + \hat{P}_z\hat{\Sigma} )(Id - \hat{S}_z \pi \hat{\Sigma})^{-1} (\hat{S}_x + \hat{S}_y\hat{\pi} + \hat{S}_z\pi \hat{\sigma}_x) \\
&-\hat{\pi} \hat{B}(Id - \hat{S}_z \pi \hat{\Sigma})^{-1} (\hat{S}_x + \hat{S}_y\hat{\pi} + \hat{S}_z\pi \hat{\sigma}_x) \\
d\eta_t &= \psi_t dt + \phi_t dW_t, \;\;\; t \geq 0.  
\label{system_ansatz_inf_bis}
\end{cases}
\end{equation}
Existence of $(\pi, \hat{\pi})$ in whole generality is a difficult problem. Let us first rewrite the system \eqref{system_ansatz_inf_bis} as:
\begin{equation}
F((\pi, \hat{\pi}), \cali{C}) = 0
\end{equation}
Where $\cali{C} = (b_x, \sigma_x, B, \Sigma ,\hat{b}_x, \hat{\sigma}_x, \hat{B}, \hat{\Sigma}, S_x, S_y, S_z, \hat{S}_x, \hat{S}_y, \hat{S}_z, P_y, P_z, P_{\c}, \hat{P}_y, \hat{P}_z, \hat{P}_{\c})$. Note that $F$ is continuously differentiable on its domain of definition. Thus, if $\entrecro{\frac{\partial F}{\partial \pi}, \frac{\partial F}{\partial \hat{\pi}}}(\pi, \hat{\pi}, \cali{C})$ is invertible for, then, by the implicit function theorem, there exists an open set $U$ containing $\cali{C}$ and a continuously differentiable function $g : U \mapsto (\R^{nd \times d})^2$ such that for all admissible coefficients $\cali{C} \in U$: $F(g(\cali{C}), \cali{C}) = 0$ and the solutions $(\pi, \hat{\pi}) = g(\cali{C})$ are unique. It means that if we find a solution to \eqref{system_ansatz_inf_bis} while the condition $\entrecro{\frac{\partial F}{\partial \pi}, \frac{\partial F}{\partial \hat{\pi}}}$ is invertible, then for small perturbations on the coefficients we still have solutions for $(\pi, \hat{\pi})$.

Let us now give sufficient conditions to ensure the existence of $(\pi, \hat{\pi})$ in a simplified setting where the state $t\to X_t$ belongs to $\R$ and all the players are symmetric in the sense that all the coefficients associated with each player are equals ($b_1 = ... = b_2, Q^1=...=Q^n, etc. $). We suppose also that the volatility is not controlled i.e. $\sigma_i = 0$ for all $i \in \inter$. In such a case $\pi, \hat{\pi} \in \R^{n\times 1}$, $\pi_1 = ... = \pi_n$, $\hat{\pi}_n = ... = \hat{\pi}_n$ and the systems \eqref{system_ansatz_inf_bis} of coupled equations now reduces to two coupled second order equations:
\begin{equation}
\begin{cases}
\pi_1^2 \entrecro{n b_1 S_{y,1}} + \pi_1 \entrecro{-b_x + P_{y,1}+P_{z,1}\sigma_x + \sum_{j\neq1} P_{\c,1,j}S_{y,1} - n b_1S_{x,1} } + \sum_{j\neq1} P_{\c,1,j} S_{x,j}  = 0 \\
\hat{\pi}_1^2 \entrecro{n \hat{b}_1 \hat{S}_{y,1}} + \hat{\pi}_1 \entrecro{-\hat{b}_x + \hat{P}_{y,1} + \sum_{j\neq1} \hat{P}_{\c,1,j} \hat{S}_{y,1} - n \hat{b}_1 \hat{S}_{x,1} } + \sum_{j\neq1} \hat{P}_{\c,1,j} \hat{S}_{x,j} + \pi_1 \hat{P}_{z,1} \hat{\sigma}_x = 0.  
\end{cases}
\label{condidtion_existence}
\end{equation}
If we note:
\bec{
a &:= n b_1 S_{y,1}\\
b &:= -b_x + P_{y,1}+P_{z,1}\sigma_x + \sum_{j\neq1} P_{\c,1,j} S_{y,1} - n b_1 S_{x,1}\\
c &:= \sum_{j\neq1} P_{\c,1,j} S_{x,j},
}
then a sufficient condition for $\pi_1$ to exists is simply: $b^2 - 4ac \geq 0$. Since $a\leq 0$ and $c \geq 0$ we have two possibilities a priori for $\pi_1$. We choose the positive one to ensure that $\cbold \in \cspace$. Then if we note:
\bec{
\hat{a} &:= n \hat{b}_1 \hat{S}_{y,1} \\
\hat{b} &:= -\hat{b}_x + \hat{P}_{y,1} + \sum_{j\neq1} \hat{P}_{\c,1,j} \hat{S}_{y,1} - n \hat{b}_1 \hat{S}_{x,1}\\
\hat{c}(\pi_1) &:= \sum_{j\neq1} \hat{P}_{\c,1,j} \hat{S}_{x,j} + \pi_1\hat{P}_{z,1} \hat{\sigma}_x,  
}
a sufficient condition for $\hat{\pi}_1$ to exist is $\hat{b}^2 - 4\hat{a}\hat{c}(\pi_1) \geq 0$. To ensure that there is a positive solution we also need $\hat{c}(\pi_1) \geq 0$.

\item Let us finally verify that $\cbstar \in \cspace$. Let us consider the candidate for the optimal control for each player:
\begin{equation*}
\begin{split}
\cbstar - \bar{\cbstar} &= A_x(X-\bar{X}) + R_t - \bar{R}_t \\
\bar{\cbstar} &= \hat{A}_x\bar{X} + \bar{\hat{R}}_t 
\end{split}
\end{equation*}
where the coefficients are defined in \eqref{feedback_control_complete} and $X^\star$ is the state process optimally controlled. Since $R, \hat{R} \in  L^{2}_{\mathbb{F}}( \Omega \times \R_+, \R^{\sum_i d_i})$ and given that the coefficient are constant in the infinite horizon case, we need to verify that:
\begin{align*}
\int_0^\infty e^{-\rho u} \E[|X^\star_u - \E[X^\star_u]|^2] du < \infty &&  \int_0^\infty e^{-\rho u} |\E[X^\star_u]|^2 du < \infty. 
\end{align*}
As we will see below, we will have to choose $\rho$ large  enough to ensure these conditions. From the above expressions we see that $X^\star$ satisfies:
\begin{equation*}
\begin{cases}
dX^\star_t &= b_t^\star dt + \sigma^\star_t dW_t \\
X_0 &= x_0
\end{cases}
\end{equation*}
with:
\begin{align*}
& b^\star_t = \beta^\star_t + B^\star(X_t^\star - \E[X_t^\star]) + \hat{B}^\star \E[X^\star_t] && \sigma^\star_t = \gamma^\star_t + \Sigma^\star(X_t^\star - \E[X_t^\star]) + \hat{\Sigma}^\star \E[X^\star_t]
\end{align*}
where we define 
\begin{align*}
&B^\star = b_x + B A_x && \hat{B}^\star = \hat{b}_x + \hat{B} \hat{A}_x && \Sigma^\star = \sigma_x + \Sigma A_x && \hat{\Sigma}^\star = \hat{\sigma}_x + \hat{\Sigma} \hat{A}_x \\
&\beta^\star_t = \beta_t + B (R_t - \E[R_t]) + \hat{B} \bar{\hat{R}}_t \\
&\gamma^\star_t = \sigma_t + \Sigma (R_t - \E[R_t]) + \hat{\Sigma} \bar{\hat{R}}_t. 
\end{align*}
By It\^o's formula we have:
\begin{equation*}
\begin{split}
\frac{d}{dt} e^{-\rho t } |\bar{X}_t|^2  &\leq e^{-\rho t } \entrepar{ - \rho |\bar{X}^\star_t|^2 + 2 (\bar{b}^\star_t)^\intercal \bar{X}_t^\star } \\
&\leq e^{-\rho t } \entrepar{  |\bar{X}^\star_t|^2 (- \rho + 2 \hat{B}^\star) + 2|\bar{X}^\star_t| |\bar{\beta}^\star_t| } \\
&\leq e^{-\rho t } \entrepar{ |\bar{X}^\star_t|^2 (- \rho + 2 \hat{B}^\star + \epsilon) + \frac{1}{\epsilon}|\bar{\beta}^\star_t|^2}. 
\end{split}
\end{equation*}
If we now set:
\begin{align*}
K = |\E[X_0^\star]|^2 + \frac{1}{\epsilon} \int_0^\infty e^{-\rho u} |\bar{\beta}_t^\star|^2  du && C= - \rho + 2 \hat{B}^\star + \epsilon,
\end{align*}
then, by Grownall inequality we obtain:
\begin{equation*}
e^{-\rho t} |\E[X_t^\star]|^2 \leq K e^{C t}.
\end{equation*}
Therefore, in order to have $\int_0^\infty e^{-\rho u}|\E[X_u^\star]|^2 \dd u < \infty$,  we shall impose  that $\rho > 2 \hat{B}^\star$.  Finally, by It\^o's formula we also have:
\begin{equation}
\begin{split}
\frac{d}{dt} \E[e^{-\rho t}|X_t^\star - \bar{X}_t^\star|^2] 
&\leq e^{-\rho t } \EE{-\rho|X_t^\star - \bar{X_t}^\star|^2 + 2 (b_t^\star - \bar{b_t}^\star)^\intercal(X_t^\star - \bar{X_t}^\star) +   |\sigma_t^\star|^2    } \\
&\leq e^{-\rho t } \E\left[|X_t^\star - \bar{X_t}^\star|^2(-\rho + 2 B^\star + \epsilon) + \frac{1}{\epsilon}|\beta^\star_t - \bar{\beta_t}^\star|^2 \right.\\
&\left. + 4(|\gamma^\star_t|^2 + |\Sigma^\star|^2 |X_t^\star - \E[X_t^\star]|^2 + |\hat{\Sigma}^\star|^2 |\E[X^\star_t]|^2) \right] \\
&\leq e^{-\rho t } \E\left[|X_t^\star - \bar{X_t}^\star|^2(-\rho + 2 B^\star + \epsilon + 4|\Sigma^\star|^2 ) + \right. 
\\& \left.\frac{1}{\epsilon}|\beta^\star_t - \bar{\beta_t}^\star|^2 +4(|\gamma^\star_t|^2 + |\hat{\Sigma}^\star|^2 |\E[X^\star_t]|^2) \right]. 
\end{split}
\end{equation}
If we now set:
\begin{align*}
&K = |X_0^\star - \bar{X}_0^\star|^2  +  \int_0^\infty e^{-\rho u} \EE{\frac{1}{\epsilon}|\beta^\star_u - \bar{\beta_u}^\star|^2+ 4(|\gamma^\star_u|^2 + |\hat{\Sigma}^\star|^2 |\E[X^\star_u]|^2)}  du  \\
&C = -\rho + 2 B^\star  + 4|\Sigma^\star|^2 + \epsilon,
\end{align*}
then, by Grownall inequality we obtain:
\begin{equation*}
e^{-\rho t} \EE{|X_t^\star - \bar{X_t}^\star|^2} \leq K e^{-C t}. 
\end{equation*}
This time in order to ensure the convergence $\int_0^\infty e^{-\rho u} \EE{|X_u^\star - \bar{X_u}^\star|^2} \dd u < \infty$,  we will add the constraint $\rho > 2 B^\star  + 4|\Sigma^\star|^2$. 
To conclude, in order to ensure that $\cbstar \in \cspace$ we make the following assumption:

\vspace{2mm}

\textbf{(H5')} $\rho > \max \entrecro{2 \hat{B}^\star,2 B^\star  + 4|\Sigma^\star|^2}$.

\end{enumerate}

\subsection{The case of common noise}

Let $W$ and $W^0$ be two independent Brownian motions defined on the same probability space $(\Omega, \mathcal{F},\mathbb{F} ,\mathbb{P})$ where $\mathbb{F} = \{\cali{F}_t\}_{t\in [0,T]}$ is the filtration generated by the pair $(W,W^0)$. Let $\mathbb{F}^0 = \{\cali{F}_t^0\}_{t \in [0,T]}$ be the filtration generated by $W^0$. For any $X_0$ and $\cbold \in \cspace$ as in Section 1, the controlled process $X^{\cbold}$ is defined by:
\bec{
d\X = b \argscond dt &+ \sigma \argscond dW_t  \\
&+ \sigma^0 \argscond dW_t^0 \\
X_0^{\c} = X_0
}
where for each $t \in [0,T]$, $x, \bar{x} \in \R^d$, $a_i, \bar{a}_i \in \R^{d_i}$:
\bec{
b(t,x,\bar{x}, \cbold, \bar{\cbold}) &= \beta_t + \bx x + \tilde{b}_{x,t}\bar{x} + \sum_{i=1}^n b_{i,t} \c_{i} + \tilde{b}_{i,t} \bar{\c}_{i}  \\
&= \beta_t + \bx x + \tilde{b}_{x,t}\bar{x} + B_t \cbold + \tilde{B}_t \bar{\cbold} \\
\sigma (t,x,\bar{x}, \cbold, \bar{\cbold})  &= \gamma_t + \sx x + \tilde{\sigma}_{x,t}\bar{x} + \sum_{i=1}^n \sigma_{i,t} \c_{i} + \tilde{\sigma}_{i,t} \bar{\c}_{i} \\
&= \gamma_t + \sx x + \tilde{\sigma}_{x,t}\bar{x} + \Sigma_t \cbold + \tilde{\Sigma}_t \bar{\cbold} \\
\sigma (t,x,\bar{x}, \cbold, \bar{\cbold})  &= \gamma_t^0 + \sx^0 x + \tilde{\sigma}_{x,t}^0\bar{x} + \sum_{i=1}^n \sigma_{i,t}^0 \c_{i} + \tilde{\sigma}^0_{i,t} \bar{\c}_{i}  \\
&= \gamma_t^0 + \sx^0 x + \tilde{\sigma}^0_{x,t}\bar{x} + \Sigma_t^0 \cbold + \tilde{\Sigma}_t^0 \bar{\cbold}. 
}
Since we will condition on $W^0$, we assume that  the coefficients $b_x, \tilde{b}_x, b_i, \tilde{b}_i, \sigma_x, \tilde{\sigma}_x, \sigma_i, \tilde{\sigma}_i, \sigma_x^0, \tilde{\sigma}_x^0, \sigma_i^0, \tilde{\sigma}_i^0$ are essentially bounded and $\mathbb{F}^0$-adapted processes, whereas $\beta, \gamma, \gamma^0$ are square integrable $\mathbb{F}^0$-adapted processes. The problem of each player $i$ 
is to minimize over $\alpha_i$ $\in$ ${\cal A}^i$, and given control $\alpha^{-i}$ of the other players,  a cost functional  of the form
\begin{equation}
\begin{split}
J\i(\alpha_i,\alpha^{-i})  = \E \left[ \int_0^T e^{-\rho t} f\i (X_t^{\cbold}, \E[X_t^{\cbold}|W_t^0], \cbold_t, \E[\cbold_t|W_t^0]) dt  + e^{-\rho T} g\i(X_T^{\c},\E[X_T^{\c}|W_T^0] )\right]
\end{split}
\end{equation}
with $f\i, g\i$ as in \eqref{coefJ}. We now suppose that $Q\i, \tilde{Q}\i,I\i, \tilde{I}\i, N\i, \tilde{N}\i $ are essentially bounded and $\mathbb{F}^0$-adapted, $L_x\i, L\i_k$ are square-integrable $\mathbf{F}$-adapted processes, $P\i, \tilde{P}\i$ are essentially bounded $\cali{F}_T^0$-mesurable random variable and $r\i$ are square-integrable $\cali{F}_T$-mesurable random variable. Hypothesis \ref{H2_c} and \ref{H2_d} of \textbf{(H2)} still holds. As in \textbf{step 1} we guess a random field of the type $\cali{W}^{\cbold,i}_t = w_t\i(X_t^{\cbold}, \EE{X_t^{\cbold}|W_t^0})$ where $w\i$ is of the form  $w_t\i(x,\bar{x}) =  \scal{K_t\i}{(x-\bar{x})} + \scal{\Lambda_t\i }{\bar{x}} +2 Y_t^{i \intercal}x + R\i_t$ with suitable coefficients $K\i, \Lambda\i, Y\i, R\i$. Given that the quadratic coefficient in $f\i, g\i$ are $\mathbb{F}^0$-adapted we guess that $K\i, \Lambda\i$ are also $\mathbb{F}^0$-adapted. Since the linear coefficients in $f\i, g\i$ and the affine coefficients in $b, \sigma, \sigma^0$ are $\mathbb{F}$-adapted, we guess that $Y\i$ is $\mathbb{F}$-adapted as well. Thus for each player we look for processes $(K\i, \Lambda\i, Y\i, R\i)$ valued in $\mathbb{S}^d_+ \times \mathbb{S}^d_+ \times \R^d\times\R$ and of the form:
\begin{equation}
\begin{cases}
dK\i_t = \dot{K}_t\i dt + Z^{K\i}_t dW_t^0 &K\i_T = P\i \\
d\Lambda\i_t = \dot{\Lambda}_t\i dt   + Z^{\Lambda\i}_t dW_t^0   &\Lambda\i_T  = P\i + \tilde{P}\i \\
dY_t\i = \dot{Y}_t\i dt + Z_t^{i} dW_t + Z^{0,Y\i}_t dW_t^0    &Y_T\i = r\i \\
dR_t\i = \dot{R}_t\i dt &R\i_T = 0
\end{cases}
\end{equation}
where $\dot{K}\i, \dot{\Lambda}\i, Z^{K\i},Z^{\Lambda\i}$ are $\mathbb{F}^0$-adapted processes valued in $\mathbb{S}^d$; $\dot{Y}\i, Z^{Y\i}, Z^{0,Y\i}$ are $\mathbb{F}$-adapted processes valued in $\R^d$ and $R\i$ are continuous functions valued in $\R$. In \textbf{step 2} we now consider, for each player $i \in \inter$, a family of processes of the form:
\begin{equation}
    \cali{S}^{\cbold,i}_t =  e^{-\rho t } w_t\i(X_t^{\cbold}, \EE{X_t^{\cbold}|W_t^0}) + \int_0^t e^{-\rho u } f\i(u,X_u^{\cbold}, \EE{X_u^{\cbold}|W_u^0}, \cbold_u, \EE{\cbold_u|W_u^0}) du.   
\end{equation}
By It\^o's formula we then obtain for \eqref{coeffX_a} and \eqref{coeffA_a}:
\begin{equation}
\begin{cases}
\Phi \i_t &= Q\i _t + \sx^\intercal K\i_t \sx + \sigma^{0\intercal}_t K\i_t \sigma^{0}_t + K\i_t \bx + \bx^\intercal K\i_t + Z^{K\i}_t \sigma_{x,t}^0 + \sigma_{x,t}^{0\intercal} Z^{K\i}_t - \rho K\i_t   \\
\Psi \i_t &= \hat{Q}\i _t + \hat{\sigma}_{x,t}^\intercal K\i_t \hat{\sigma}_{x,t} +\hat{\sigma}_{x,t}^{0\intercal} \Lambda\i_t \hat{\sigma}_{x,t}^0 + \Lambda\i_t \hat{b}_{x,t} + \hat{b}_{x,t}^\intercal \Lambda\i_t + Z^{\Lambda\i}_t \Hat{\sigma}_{x,t}^0 + \Hat{\sigma}_{x,t}^{0\intercal} Z^{\Lambda\i}_t - \rho \Lambda\i_t \\   
\Delta \i_t & = L\i_{x,t} + \bx^\intercal Y\i_t + \bxt^\intercal \bar{Y}\i_t + \sx^\intercal Z\i_t + \sxt^\intercal \bar{Z}\i_t + \sigma_{x,t}^{0\intercal} Z^{0,Y\i}_t + \tilde{\sigma}_{x,t}^{0\intercal} \bar{Z^{0,Y\i}}_t + \Lambda\i_t \bar{\beta}_t \\
&+ \sx^\intercal K\i_t \gamma_t + \sxt^\intercal K\i_t \bar{\gamma}_t + K\i_t(\beta_t - \bar{\beta}_t) + Z^{K\i}_t(\gamma^0_t - \bar{\gamma}^0_t) + Z^{\Lambda\i}_t \bar{\gamma}^0_t
- \rho Y\i_t\\
&+ \sum_{k \neq i} U^{i\intercal}_{k,t} (\c_{k,t} - \bar{\c}_{k,t}) + V^{i\intercal}_{k,t} \bar{\c}_{k,t} \\
\Gamma\i_t &= \gamma_t^\intercal K\i_t \gamma_t + \bar{\gamma}_t^{0\intercal} \Lambda\i_t \bar{\gamma}_t^0 +(\gamma_t^0 - \bar{\gamma}^0_t)^\intercal K\i_t (\gamma_t^0 - \bar{\gamma}^0_t) +  2 \beta_t^\intercal Y\i_t + 2 \gamma_t^\intercal Z\i_t +  2 \gamma_t^{0\intercal} Z^{0,Y\i}_t \\
&+ \sum_{k \neq i} (\c[_{k,t}] - \bar{\c}_{k,t})^\intercal S\i_{k,t} (\c[_{k,t}] - \bar{\c}_{k,t}) + \bar{\c}_{k,t}^\intercal \hat{S}\i_{k,t} \bar{\c}_{k,t} +2 [O^{i}_{k,t} + \xi^{i}_{k,t} - \bar{\xi}^{i}_{k,t}]^\intercal \c[_{k,t}] - \rho R\i_t,
\label{coeffX_a_CM}
\end{cases}
\end{equation}
with 
\begin{equation}
\begin{cases}
S\i_{k,t} &= N\i_{k,t} + \sigma_{k,t}^\intercal K\i_t \sigma_{k,t} + (\sigma_{k,t}^0)^\intercal K\i_t \sigma_{k,t}^0  \\
\hat{S}\i_{k,t} &= \hat{N}\i_{k,t} + \hat{\sigma}_{k,t}^\intercal K\i_t  \hat{\sigma}_{k,t} + (\hat{\sigma}_{k,t}^0)^\intercal \Lambda\i_t  \hat{\sigma}_{k,t}^0 \\
U^{i}_{k,t} &= I_{k,t}\i + \sigma_{k,t}^\intercal K\i_t \sigma_{x,t} + (\sigma_{k,t}^0)^\intercal K\i_t \sigma_{x,t}^0 + (\sigma_{k,t}^0)^\intercal Z^{K\i}_t + b_{k,t}^\intercal K\i_t \\
V^{i}_{k,t} &= \hat{I}_{k,t}\i + \hat{\sigma}_{k,t}^\intercal  K\i_t \hat{\sigma}_{x,t} + (\Hat{\sigma}_{k,t}^0)^\intercal \Lambda\i_t \Hat{\sigma}_{x,t}^0 +(\Hat{\sigma}_{k,t}^0)^\intercal Z^{\Lambda\i}_t + \hat{b}_{k,t}^\intercal \Lambda \i_t \\
O^{i}_{k,t} &= \bar{L}\i_{k,t} + \hat{b}_{k,t}^\intercal \bar{Y}\i_t + \hat{\sigma}_{k,t}^\intercal \bar{Z}\i_t + \hat{\sigma}_{k,t}^\intercal K\i_t \bar{\gamma}_t + (\hat{\sigma}_{k,t}^0)^\intercal \Lambda\i_t \bar{\gamma}_t^0 + (\hat{\sigma}_{k,t}^0)^\intercal \bar{Z^{0,Y\i}_t} \\ 
&+ \frac{1}{2} \sum_{k \neq i} (\hat{J}\i_{i,k,t} + \hat{J}_{k,i,t}^{i\intercal})\bar{\c}_{k,t} \\
J\i_{k,l,t} &= G\i_{k,l,t} +  \sigma_{k,t}^\intercal K\i_t \sigma_{l,t} + (\sigma_{k,t}^0)^\intercal K\i_t \sigma_{l,t}^0  \\
\hat{J}\i_{k,l,t} &= \hat{G}\i_{k,l,t} +  \hat{\sigma}_{k,t}^\intercal K\i_t \hat{\sigma}_{l,t} + (\hat{\sigma}_{k,t}^0)^\intercal \Lambda\i_t  \hat{\sigma}_{l,t}^0  \\
\xi^{i}_{k,t} &= L\i_{k,t} + b_{k,t}^\intercal Y\i_t + \sigma_{k,t}^\intercal Z\i_t + (\sigma_{k,t}^0)^\intercal Z^{0,Y\i}_t + \sigma_{k,t}^\intercal K\i_t \gamma_t + (\sigma_{k,t}^0)^\intercal K\i_t \gamma_t^0 \\
&+ \frac{1}{2} \sum_{k \neq i} (J\i_{i,k,t} + J_{k,i,t}^{i\intercal})\c_{k,t}. 
\label{coeffA_a_CM}
\end{cases}
\end{equation}
Note that we now denote by $\bar{U}$ the conditional expectation with respect to $W_t^0$, i.e. $\bar{U} = \EE{U|W_t^0}$. Then, at \textbf{step 3}, we constraint the coefficients $(K\i, \Lambda\i, Y\i, R\i)$ to satisfy the following problem:
\begin{equation}
\begin{cases}
dK\i_t = -\Phi^{i0}_t dt + Z^{K\i}_t dW_t^0 &K\i_T = P\i \\
d\Lambda\i_t = -\Psi^{i0}_t dt + Z^{\Lambda\i}_t dW_t^0    &\Lambda\i_T = P\i + \tilde{P}\i \\
dY_t\i = -\Delta^{i0}_t dt + + Z_t^{i} dW_t + Z^{0,Y\i}_t dW_t^0 &Y_T\i = r\i \\
dR_t\i = -\Gamma_t^{i0} dt &R\i_T = 0
\end{cases}
\label{sys_coeff_optimControl_CM}
\end{equation}
where $\Phi^{i0}, \Psi^{i0}, \Delta^{i0}, \Gamma^{i0}$ are defined in \eqref{sys_coeff_optimControl}. Thus we obtain the best response functions of the players:
\begin{equation}
\begin{split}
\c_{i,t} 
&=- {S}^{i -1}_{i,t} {U}^{i }_{i,t} (X_t -\E[X_t |W_t^0]) - {S}^{i -1}_{i,t} ({\xi}\i_{i,t} - {\bar{\xi}}\i_{i,t}) -\hat{S}_{i,t}^{i -1}(V^{i}_{i,t} \E[X_t|W_t^0] + O^{i}_{i,t}). 
\end{split}
\label{best_response_CM}
\end{equation}
We then proceed to \textbf{step 5} and to the search of a fixed point in the space of controls. The only difference at that point is in the ansatz for $t \mapsto \boldmath{Y_t}$. Since we consider the case of common noise,  we now search for an ansatz of the form $\boldsymbol{Y}_t = \pi_t (X_t - \bar{X}_t) + \hat{\pi}_t \bar{X}_t + \eta_t $
where $(\pi, \hat{\pi}, \eta) \in L^{\infty}([0,T], \R^{nd \times d}) \times L^{\infty}([0,T], \R^{nd \times d}) \times \cali{S}_\F^2(\Omega\times [0,T], \R^{nd})$ satisfy:
\begin{equation}
\begin{cases}
d\eta_t &= \psi_t dt + \phi_t dW_t + \phi_t^0 dW_t^0 \\
\eta_T &= \boldsymbol{r} = (r^i)_{i\in \inter} \\
d\pi_t &= \dot{\pi}_t dt + Z_t^{0,\pi\i} dW_t^0\\
\pi_T &= 0 \\
d\hat{\pi}_t &= \dot{\hat{\pi}}_t dt + Z_t^{0,\Hat{\pi}\i} dW_t^0\\
\hat{\pi}_T &=0. 
\end{cases}
\label{sys_fixed_point_CM}
\end{equation}
The method to determine the coefficients is then similar. Existence and uniqueness of a solution $(K\i, \Lambda\i)$ to the backward stochastic Ricatti equation in \eqref{sys_coeff_optimControl_CM} is discussed in \cite{Pham2016}, section 3.2. The existence of a solution $(Y\i, Z\i, Z^{0,Y\i})$ to the linear mean-field BSDE in \eqref{sys_coeff_optimControl_CM} is obtained as in \textbf{step 6} thanks to Thm. 2.1 in \cite{Li2017}. As in the previous section the existence of a soltion  $(\pi, \hat{\pi}) \in L^{\infty}_{\mathbb{F}^0}( \Omega \times \R_+, \R^{nd\times d})$ (essentially bounded $\mathbb{F}^0$-adapted functions) of \eqref{sys_fixed_point_CM} in the general case is a conjecture and needs to be verified in each example. We are not aware of any work tackling the existence of solutions in such situation. Given $(\pi, \hat{\pi})$ the existence of $(\eta, \phi, \phi^0)$ solution to \eqref{sys_fixed_point_CM} is ensured as in the previous section by Thm. 2.1 in \cite{Li2017}.

\subsection{The case of multiple Brownian motions}

We quickly sketch an extension to the case where there are multiple Brownian motions driving the state equation. The assumptions on the coefficients are the same as in the previous part. Only the length of the calculus changes. Let us now consider the state dynamic:
\begin{equation*}
d\X = b \args dt + \sum_{\ell=1}^\kappa \sigma^\ell \args dW_t^\ell
\end{equation*}
where $\Phi^{i0}, \Psi^{i0}, \Delta^{i0}, \Gamma_t^{i0}$ are defined in \eqref{sys_coeff_optimControl}.
\begin{equation}
\begin{cases}
\cbold_t &= (\c_{1,t},..., \c_{n,t}) \\ 
b(t,x,\bar{x}, \boldsymbol{a}, \boldsymbol{\bar{a}}) &= \beta_t + \bx x + \tilde{b}_{x,t} \bar{x} + \sum_{i=1}^n b_{i,t} a_{i,t} + \tilde{b}_{i,t} \bar{a}_{i,t}  \\
\sigma^\ell (t,x,\bar{x}, \boldsymbol{a}, \boldsymbol{\bar{a}}) &= \gamma_t^\ell + \sx^\ell x + \tilde{\sigma}_{x,t}^\ell\bar{x} + \sum_{i=1}^n \sigma_{i,t}^\ell \c_{i,t} + \tilde{\sigma}_{i,t}^\ell \bar{\c}_{i,t}.  
\end{cases}
\label{coef_MMB}
\end{equation}
We require the coefficients in \eqref{coef_MMB} to satisfy an adaptation of \textbf{(H1)} where $\gamma, \sigma_x, \tilde{\sigma}_x, (\sigma_i, \tilde{\sigma}_i)_{i \in \inter}$ are replaced by 
$\gamma^\ell, \sigma_x^\ell, \tilde{\sigma}_x^\ell, (\sigma_i^\ell, \tilde{\sigma}_i^\ell)_{i \in \inter}$ for $\ell \in \{1,...,\kappa\}$.

To take into account the multiple Brownian motions in \textbf{step 1}, we now search for random fields of the form $w_t\i(x,\bar{x}) = \scal{K_t\i}{(x-\bar{x})} + \scal{\Lambda_t\i }{\bar{x}} +2 Y_t^{i\intercal}x + R_t\i$ with the processes $(K\i,\Lambda\i, Y\i, (Z^{\ell,i})_{\ell\in \llbracket 1,\kappa \rrbracket}, R\i)$ in 
$(L^{\infty}([0,T], \mathbb{S}^d_+))^2\times,\cali{S}^2_{\mathbb{F}}(\Omega\times [0,T],\R^{d}) \times (L^2_{\mathbb{F}}(\Omega\times [0,T],\R^{d}))^\kappa 
\times L^{\infty}([0,T], \R)$ and solution to:
\begin{equation}
\begin{cases}
dK\i_t = \dot{K}_t\i dt &K\i_T = P\i \\
d\Lambda\i_t = \dot{\Lambda}_t\i dt &\Lambda\i_T = P\i + \tilde{P}\i \\
dY_t\i = \dot{Y}_t\i dt + \sum_\ell Z_t^{i,\ell} dW_t^\ell &Y_T\i = r\i \\
dR_t\i = \dot{R}_t\i dt &R\i_T = 0
\end{cases}
\end{equation}
where $(\dot{K}\i, \dot{\Lambda}\i, \dot{R}\i)$ are deterministic processes valued in $\mathbb{S}^d_+ \times \mathbb{S}^d_+ \times \R$ and $(\dot{Y}\i, Z\i)$ are adapted processes valued in $\R^d$.

The method then follows the same steps with generalized coefficients and at \textbf{step 2} we obtain generalized coefficient for \eqref{coeffX_a} and \eqref{coeffA_a}:
\begin{equation}
\begin{cases}
\Phi \i_t &= Q\i _t  + K\i_t \bx + \bx^\intercal K\i_t  + \sum_{r} \sx^{r\intercal} K\i_t \sx^r - \rho K^{i}_t\\
\Psi \i_t &= \hat{Q}\i _t  + \Lambda\i_t \hat{b}_{x,t} + \hat{b}_{x,t}^\intercal \Lambda\i_t + \sum_{r} \hat{\sigma}_{x,t}^{r\intercal} \Lambda\i_t \hat{\sigma}_{x,t}^r - \rho \Lambda_t^{i} \\
\Delta \i_t & = L\i_{x,t} + \bx^\intercal Y\i_t + \bxt^\intercal \bar{Y}\i_t + \Lambda\i_t \bar{\beta}_t + K\i_t(\beta_t - \bar{\beta}_t) \\
&  + \sum_\ell \sx^{\ell\intercal} K\i_t \gamma^\ell_t + \sxt^{\ell\intercal} K\i_t \bar{\gamma}^\ell_t + \sx^{\ell\intercal} Z^{i,\ell}_t + \sxt^{\ell\intercal} \bar{Z}^{i,\ell}_t 
- \rho Y^{i}_t \\
&+ \sum_{k \neq i} U\i_{k,t} (\c_{k,t} - \bar{\c}_{k,t}) + V\i_{k,t} \bar{\c}_{k,t} \\
\Gamma\i_t &= 2 \beta_t^\intercal Y\i_t + \sum_\ell \gamma_t^{\ell\intercal} K\i_t \gamma_t^\ell +  2 \gamma_t^{\ell\intercal} Z^{i,\ell}_t\\
&+ \sum_{k \neq i} (\c[_{k,t}] - \bar{\c}_{k,t})^\intercal S\i_{k,t} (\c[_{k,t}] - \bar{\c}_{k,t}) + \bar{\c}_{k,t}^\intercal \hat{S}\i_{k,t} \bar{\c}_{k,t} +2 [O^{i}_{k,t} + \xi^{i}_{k,t} - \bar{\xi}^{i}_{k,t}]^\intercal \c[_{k,t}] 
\end{cases}
\label{coeffX_extended}
\end{equation}

\begin{equation}
\begin{cases}
S\i_{k,t} &= N\i_{k,t} + \sum_\ell \sigma_{k,t}^{\ell\intercal} K\i_t \sigma_{k,t}^\ell  \\
\hat{S}\i_{k,t} &= \hat{N}\i_{k,t} + \sum_\ell  \hat{\sigma}_{k,t}^{\ell\intercal} K\i_t  \hat{\sigma}_{k,t}^\ell \\
U^{i}_{k,t} &= I_{k,t}\i + b_{k,t}^\intercal K\i_t + \sum_\ell \sigma_{k,t}^{\ell\intercal} K\i_t \sigma_{x,t}^\ell  \\
V^{i}_{k,t} &= \hat{I}_{k,t}\i + \hat{b}_{k,t}^\intercal \Lambda \i_t  + \sum_\ell \hat{\sigma}_{k,t}^{\ell\intercal}  K\i_t \hat{\sigma}_{x,t}^\ell \\
O^{i}_{k,t} &= \bar{L}\i_{k,t} + \hat{b}_{k,t}^\intercal \bar{Y}\i_t + \sum_\ell \hat{\sigma}_{k,t}^{\ell\intercal} \bar{Z}^{i,\ell}_t + \hat{\sigma}_{k,t}^{\ell\intercal} K\i_t \bar{\gamma}_t^\ell \\
J\i_{k,l,t} &= G\i_{k,l,t} + \sum_\ell  \sigma_{k,t}^{\ell\intercal} K\i_t \sigma_{l,t}^\ell \\
\hat{J}\i_{k,l,t} &= \hat{G}\i_{k,l,t} + \sum_\ell  \hat{\sigma}_{k,t}^{\ell\intercal} K\i_t \hat{\sigma}_{l,t}^\ell \\
\xi^{i}_{k,t} &= L\i_{k,t} + b_{k,t}^\intercal Y\i_t + \sum_\ell \sigma_{k,t}^{\ell\intercal} Z^{i,\ell}_t + \sigma_{k,t}^{\ell\intercal} K\i_t \gamma_t^\ell. 
\end{cases}
\label{coeffA_extended}
\end{equation}
From these extended formulas we can then constrain the coefficients as in \textbf{step 3} and obtain \eqref{sys_coeff_optimControl} with now the generalized coefficients defined in \eqref{coeffX_extended} and \eqref{coeffA_extended}. The \textbf{step 4} is then straightforward and we obtain the best response functions:
\begin{equation}
\begin{split}
\c_{i,t} &= a_t^{i,0}(X_t, \E[X_t]) + a_t^{i,1}(\E[X_t]) \\
&=- ({S}^{i }_{i,t})^{-1} {U}^{i }_{i,t} (X_t-\E[X_t]) - ({S}^{i }_{i,t})^{-1} ({\xi}\i_{i,t} - {\bar{\xi}}\i_{i,t}) - (\hat{S}_{i,t}^{i })^{-1} (V^{i}_{i,t} \E[X_t] + O^{i}_{i,t}). 
\end{split}
\label{best_response_MBM}
\end{equation}
From \textbf{step 4} we can then continue to \textbf{step 5}, i.e. the fixed point search. The only difference at that point is in the ansatz for$t \mapsto \boldmath{Y_t}$. Since we consider the case with multiple Brownian motions we now search for an ansatz of the form $\boldsymbol{Y}_t = \pi_t (X_t - \bar{X}_t) + \hat{\pi}_t \bar{X}_t + \eta_t $ 
where $(\pi, \hat{\pi}, \eta) \in L^{\infty}([0,T], \R^{nd \times d}) \times L^{\infty}([0,T], \R^{nd \times d}) \times \cali{S}_\F^2(\Omega\times [0,T], \R^{nd})$ satisfy:
\begin{equation}
\begin{cases}
d\eta_t &= \psi_t dt + \sum_\ell \phi_t^\ell dW_t^\ell  \\
\eta_T &= \boldsymbol{r} = (r^i)_{i\in \inter} \\
d\pi_t &= \dot{\pi}_t dt \\
\pi_T &= 0 \\
d\hat{\pi}_t &= \dot{\hat{\pi}}_t dt \\
\hat{\pi}_T &=0. 
\end{cases}
\end{equation}
The method to determine the coefficients $\pi, \hat{\pi}, \eta$ is then similar. The validity of the computations i.e. \textbf{Step 6} can be done exactly as in the case of a single brownian motion.


\section{Example} \label{secex}

We now focus on a toy example to illustrate the previous results. Let us  consider a two player game where the state dynamics is simply a Brownian motion that two players can control. 
The goal of each player is to get the state near its own target $t \mapsto T\i_t$, where $t \mapsto T\i_t$, $i$ $=$ $1,2$,  is a stochastic process. In order to add mean-field terms we suppose that each player try also to minimize the variance of the state and the variance of their controls. 
\bec{
dX_t &= ( b_1\c_{1,t} +b_2\c_{2,t} )dt + \sigma dW_t \\
J\i(\c_1, \c_2) &= \EE{\int_0^\infty e^{-\rho u}\Big( \lambda\i \var(X_u)+\delta\i (X_u-T\i_u)^2 + \theta\i \var(\c_{i,u}) + \xi\i \c_{i,u}^2 \Big) du   }, \; i=1,2, 
\label{example_1}
}
where $(\lambda\i, \delta\i, \theta\i, \xi\i) \in \R^4_+$. In order to fit to the context described in the first section we rewrite the cost function as follows:
\begin{equation}
\begin{split}
J\i(\c_1, \c_2) = \E \Big[  \int_0^\infty &e^{-\rho u} \Big( (\lambda\i + \delta\i) (X_u - \bar{X}_u)^2 + \delta\i \bar{X}_u^2 + (\theta\i + \xi)^2 (\c_{i,u} - \bar{\c}_{i,u}) \\
&  + \xi\i \bar{\c}_{i,u}^2 + 2X_u[-2\delta\i T\i] + \delta\i (T\i_u)^2 \Big) du \Big].  
\end{split}
\end{equation}
Since the terms $\delta\i (T\i)^2$ do not  influence the optimal control of the players,  we work with the slightly simplified cost function: 
\begin{equation*}
\tilde{J}\i(\c_1, \c_2) = \EE{\int_0^\infty e^{-\rho u}\Big( (\lambda\i + \delta\i) (X_u - \bar{X}_u)^2 + \delta\i \bar{X}_u^2 + (\theta\i + \xi)^2 (\c_{i,u} - \bar{\c}_{i,u}) + \xi\i \bar{\c}_{i,u}^2 + 2X_u[-2\delta\i T\i] 
\Big) du  }. 
\end{equation*}

Following the method explained in the previous section,  we use Theorem \ref{theorem_nash_eq} in order to find a Nash equilibrium. We obtain the feedback form of the open loop controls and the dynamics of the state:
\bec{
\c\i - \bar{\c}\i &= - \frac{P\i}{b_i} \entrepar{ (K\i + \pi\i) (X_t-\bar{X}_t) + \eta\i_t -\bar{\eta}\i_t } \\
\bar{\c}\i &= - \frac{\tilde{P}\i}{b_i} \entrepar{ (\Lambda\i + \tilde{\pi}\i) \bar{X}_t + \bar{\eta}\i_t } \\
\bar{X}_t &= \bar{X}_0 e^{-\tilde{a}t} + \int_0^t e^{-\tilde{a}(t-u)} \bar{\gamma}_u  du \\
X_t - \bar{X}_t &= (X_0 - \bar{X}_0)e^{-at} + \int_0^t e^{-a(t-u)} \entrecro{(\gamma_u - \bar{\gamma}_u)\dd u + \sigma dW_u} 
\label{state_control_exampel}
}
where $K\i \in \R_+, \Lambda\i \in \R_+, a \in \R, \tilde{a} \in \R, \pi \in \R^2, \tilde{\pi} \in \R^2, \eta \in L^2_{\mathbb{F}}((0,\infty), \R^2), \gamma \in L^2_{\mathbb{F}}((0,\infty), \R^2) $ satisfy:
\begin{align}
  \begin{cases}
      K\i &= \frac{- \rho + \sqrt{\rho^2 + 4 P\i(\lambda\i + \delta\i)} }{2P\i} \\
      \Lambda\i &= \frac{- \rho + \sqrt{\rho^2 + 4 \tilde{P}\i\delta\i} }{2\tilde{P}\i} \\
      a &= \sum_{i=1}^2 P\i(K\i + \pi\i) \\
\tilde{a} &= \sum_{i=1}^2 \tilde{P}\i(\Lambda\i + \tilde{\pi}\i) \\
0 &= P_y \pi - (\pi B - P_{\c})(S_x+ S_y\pi) \\
0 &= \tilde{P}_y \tilde{\pi} - (\tilde{\pi} B - \tilde{P}_{\c})(\tilde{S}_x+ \tilde{S}_y\tilde{\pi}) \\
\eta_t - \bar{\eta}_t &= - \int_t^\infty e^{[P_y - (\pi B - P_{\c})S_y](t-u)} \EE{H_u - \bar{H}_u | \mathcal{F}_t} \dd u \\
\bar{\eta}_t &= - \int_t^\infty e^{[\tilde{P}_y - (\tilde{\pi} B - \tilde{P}_{\c})\tilde{S}_y](t-u)} \bar{H}_u  \dd u \\
\gamma_t - \bar{\gamma}_t &= - \sum_{i=1}^2 P\i (\eta_{i,t} - \bar{\eta}_{i,t}) \\
\bar{\gamma}_t &= - \sum_{i=1}^2 \tilde{P}\i \bar{\eta}_{i,t} \\
  \end{cases}
  &&
  \begin{cases}
   S_x &= -( P^1 K^1/b_1, P^2 K^2/b_2) \\
\tilde{S}_x &= -( \tilde{P}^1 \Lambda^1/b_1, \tilde{P}^2 \Lambda^2/b_2)\\
S_y &= -(\1_{i=j} P\i/b_i)_{i,j \in \{1,2\} }\\
\tilde{S}_y &= -(\1_{i=j} \tilde{P}\i/b_i)_{i,j \in \{1,2\} }\\
P_{\c} &= - (\1_{i \neq j} K\i b_j)_{i,j \in \{1,2\} }\\
\tilde{P}_{\c}&= - (\1_{i \neq j} \Lambda\i b_j)_{i,j \in \{1,2\} }\\
P_y &=  (\1_{i=j} (P\i K\i + \rho))_{i,j \in \{1,2\} }\\
\tilde{P}_y &=  (\1_{i=j} (\tilde{P}\i \Lambda\i + \rho))_{i,j \in \{1,2\} }\\
H &= (\delta^1 T^1, \delta^2 T^2) \\
B &= (b_1, b_2) \\
P\i &= \frac{b_i^2}{\theta_i + \xi^i} \\
\tilde{P}\i &= \frac{b_i^2}{  \xi\i }.   
  \end{cases}
\label{coeff_example}
\end{align}



From \eqref{state_control_exampel} and \eqref{coeff_example} we can study and simulate the influence of the different parameters of the cost-function of the first player. We notice that $(\lambda^1, \theta^1)$ only influence $X - \EE{X}$ and the feedback form of $\c_1 - \bar{\c}_1$ (zero-mean terms).  
\begin{itemize}
\item If $\lambda^1 \to \infty$ then $\pi_1 \sim \lambda^1$ and $K^1 \to \infty$ which implies that $X_t - \EE{X}_t \to 0$ for all $t \geq 0$. This is expected since the term $\lambda^1$ penalizes the variance of the state $\var{(X_t)}$ in the cost function of the first player. See  Figure~\ref{fig_test_variance}.

\item If $\delta^1 \to \infty$ then $(\pi, \tilde{\pi}) \to \infty$ and $K^1 \sim \delta^1$ and $\Lambda^1 \sim \delta^1$ which imply that  $X_t  \to T^1_t$ for all $t \geq 0$. This is also expected since the term $\delta^1$ penalizes the quadratic gap between the state $X$ and the target $T^1$. See  Figure~\ref{fig_mean}.

\item If $\theta^1 \to \infty$ then $P^1 \to 0$, $P^1 K^1 \to 0$, $P^1 \pi^1 \to 0$ and  $P^1 \eta^1_t \to 0$ for every $t \geq 0$. We then have $\c_t - \bar{\c}_t \to 0$ for all $t \geq 0$ and all the terms relative to the first player in $X-\EE{X}$ disappear. Given that $\theta^1$ penalizes the variance of the control of the first player, this convergence is also intuitive.

\item If $\xi^1 \to \infty$ then $(P^1,\tilde{P}^1) \to 0$ which imply that $(\c_{1,t}, \bar{\c}_{1,t}) \to 0$ for all $t \geq 0$ and all the terms relative to the first player in $X_t-\EE{X_t}$ and $\EE{X_t}$ disappear for all $t \geq 0$. This means that the first player becomes powerless.
\end{itemize}

\begin{figure}
\centering
\subfloat[$\lambda\i=0$]{
    \label{fig_11}
    \includegraphics[width=85mm,height=4.2cm]{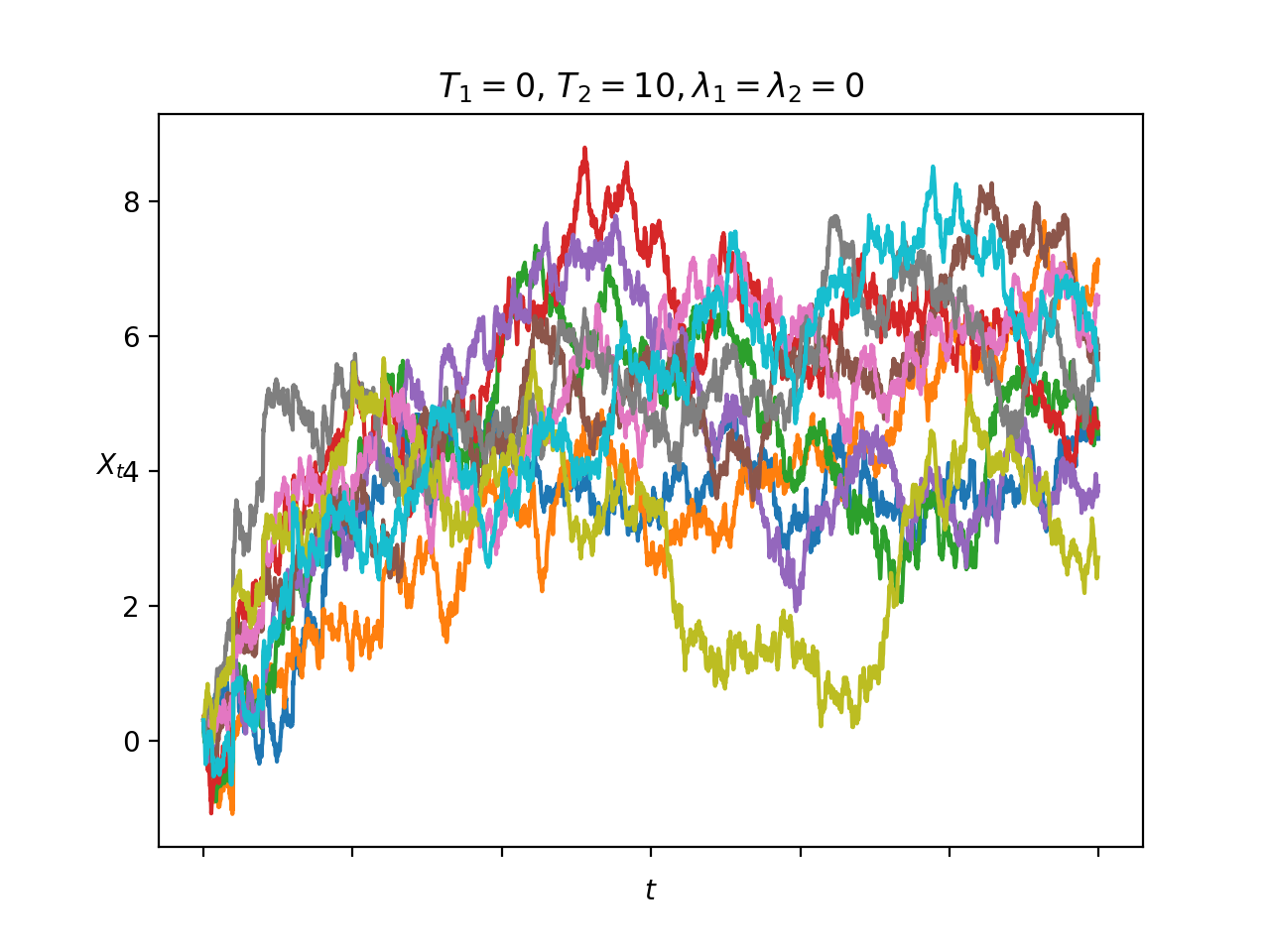}
 }
\subfloat[$\lambda\i=10$]{
    \label{fig_12}
    \includegraphics[width=85mm,height=4.2cm]{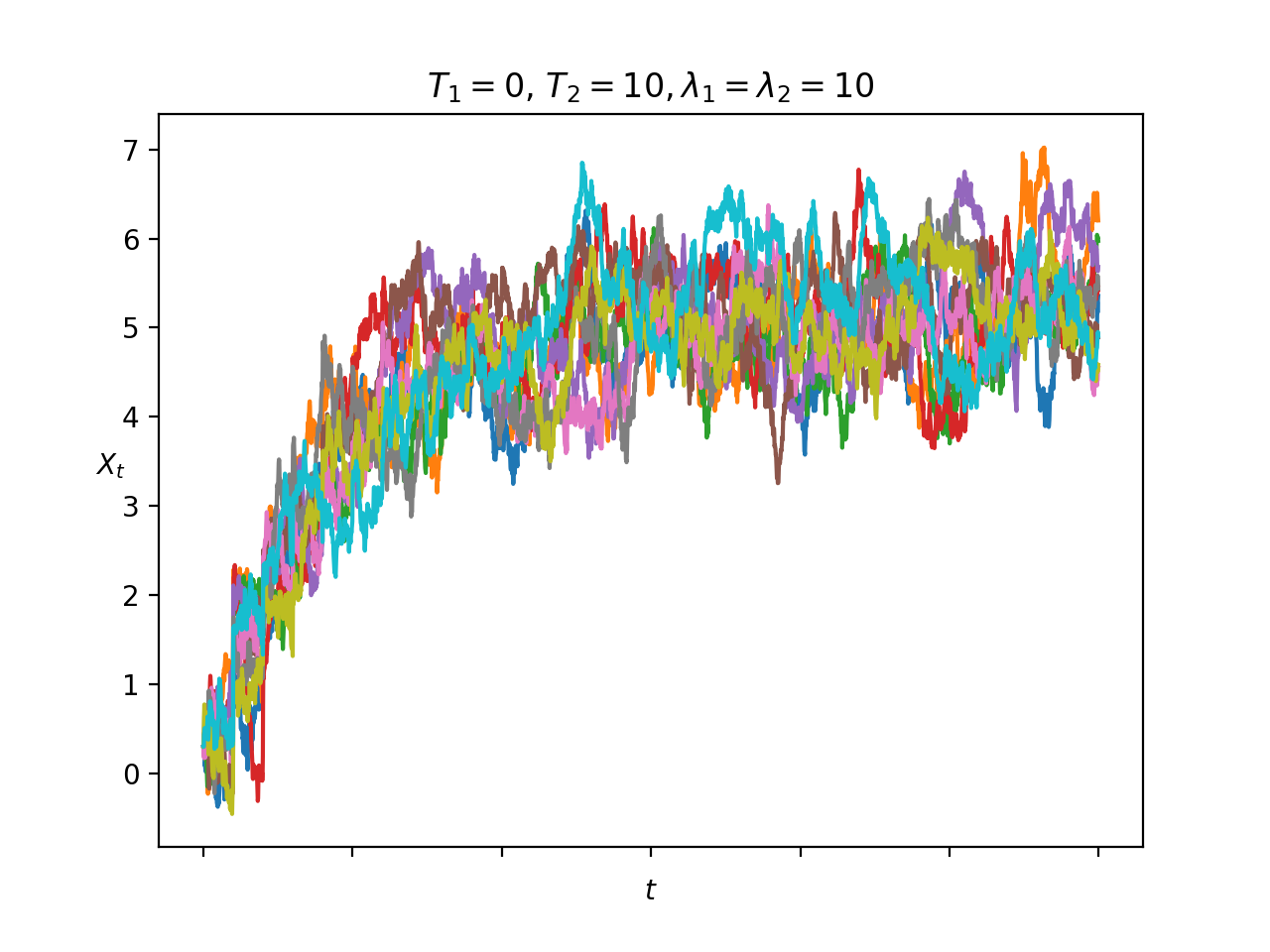}
  }
  
\subfloat[$\lambda\i=100$]{
    \label{fig_21}
    \includegraphics[width=85mm,height=4.2cm]{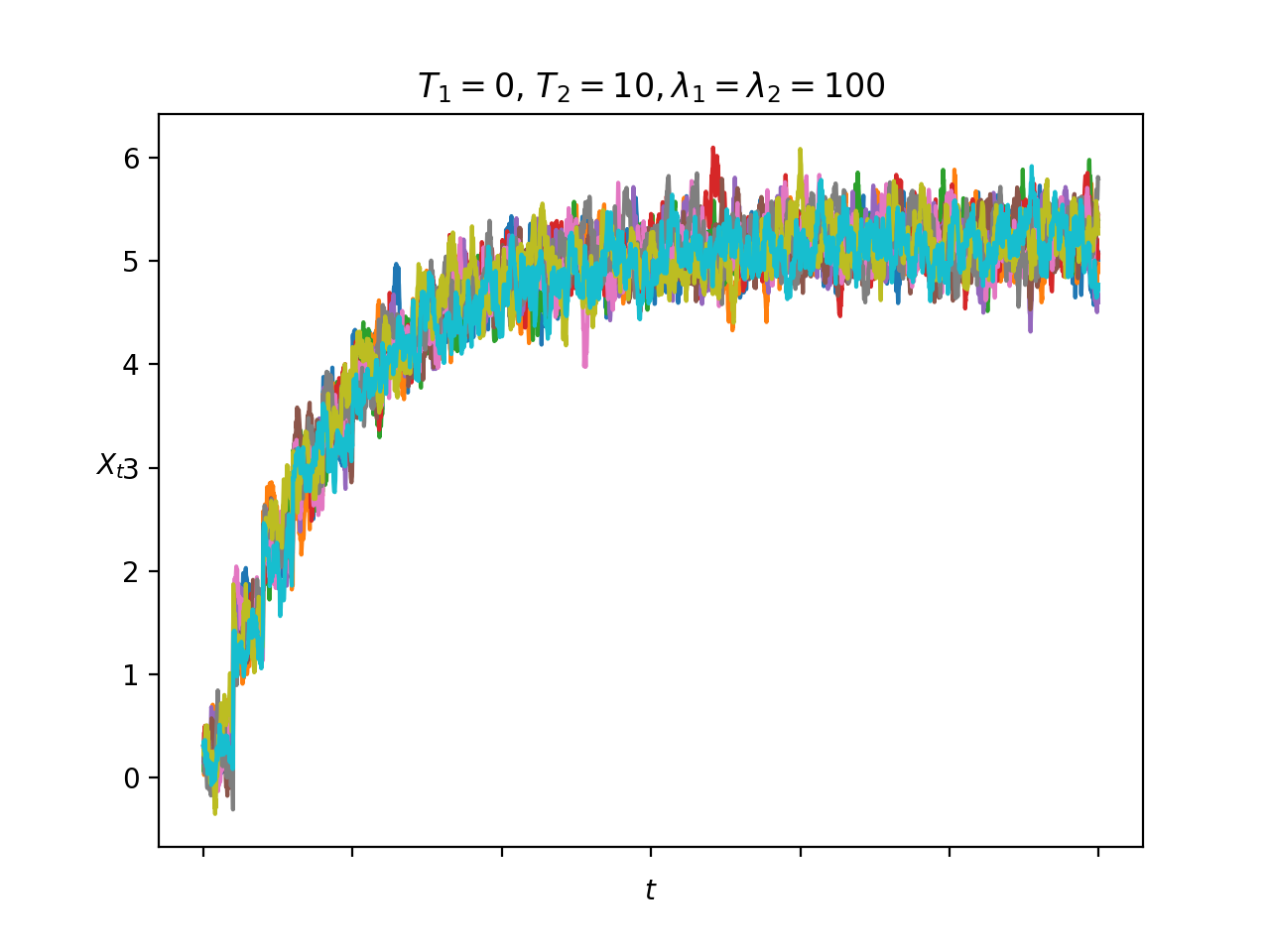}  }
\subfloat[$\lambda\i=500$]{
    \label{fig_22}
    \includegraphics[width=85mm,height=4.2cm]{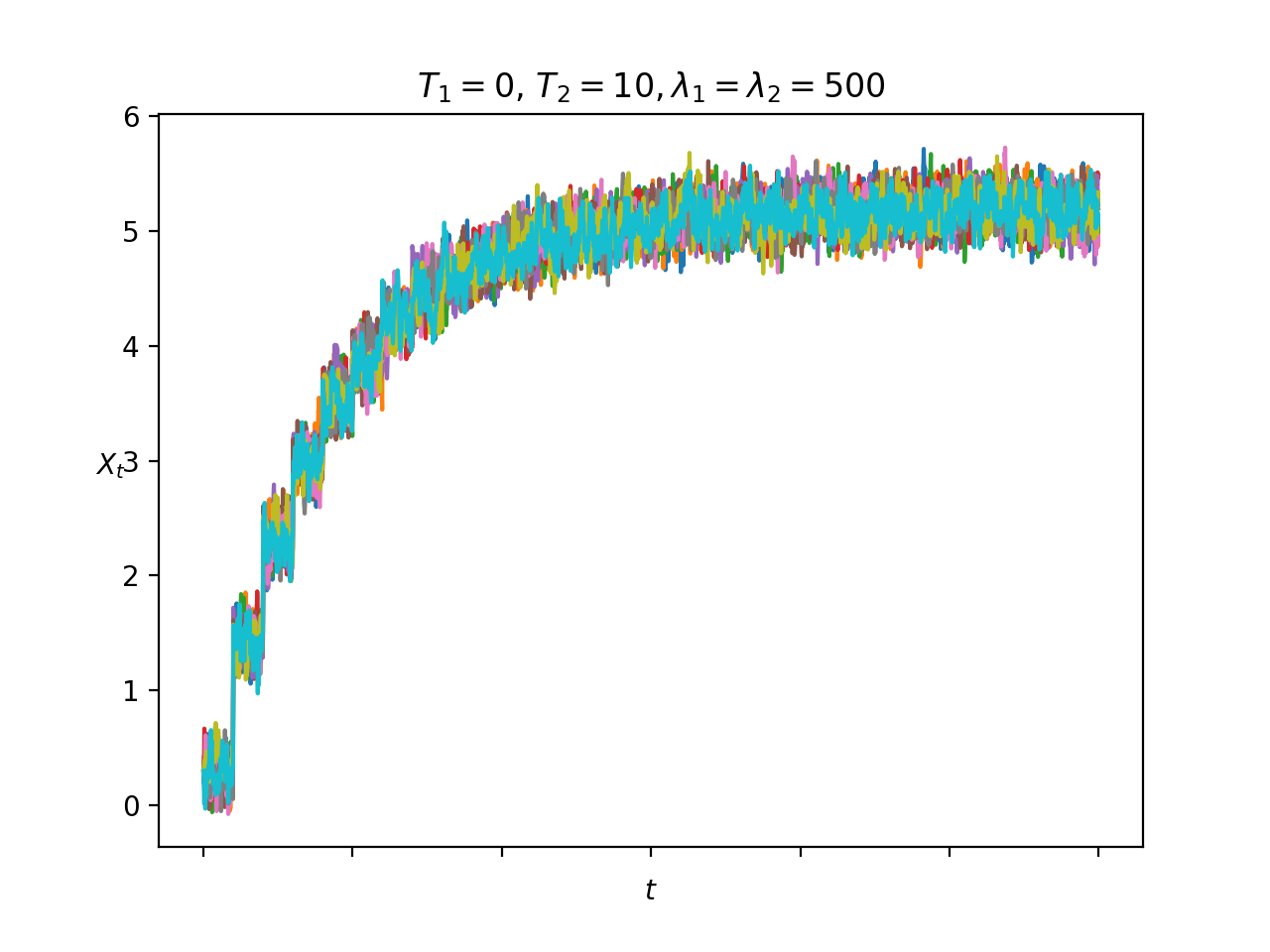}
  }
  \caption{\small{Nash equilibrium with: \\$b_i=\sigma=\delta\i=\theta\i=\xi\i=1, \rho=3, T^1=0, T^2=10$}}
\label{fig_test_variance}
\end{figure}

\begin{figure}[t!]
\centering{
    \includegraphics[width=85mm,height=4.3cm]{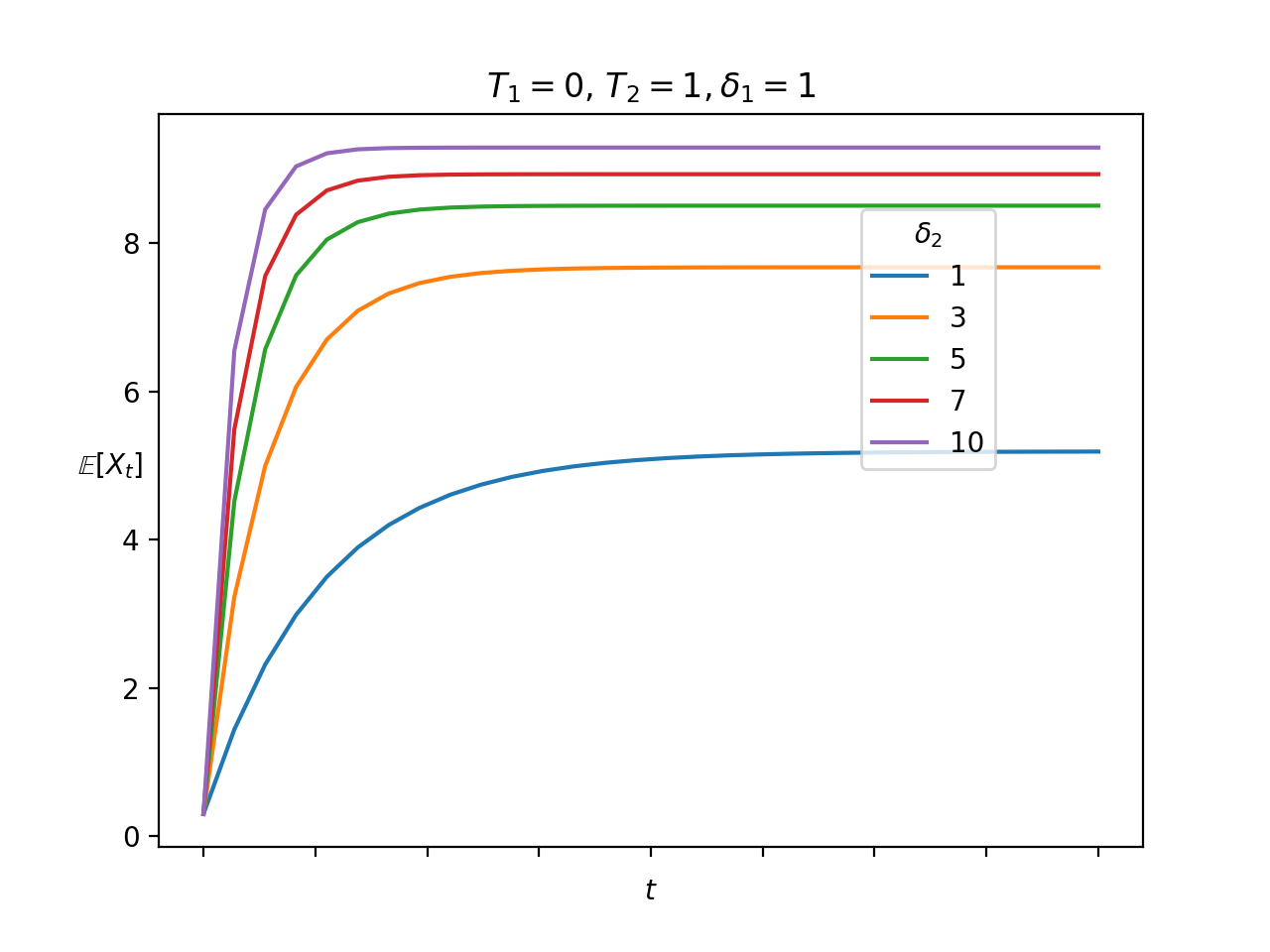}  }
\caption{\small{$t \mapsto \E[X_t]$ \\
$b_i=\sigma=\delta\i=\theta\i=\xi\i=1, \rho=3, T^1=0, T^2=10$}}
\label{fig_mean}
\end{figure}



\FloatBarrier

\vspace{13mm}




\printbibliography
\end{document}